\documentclass[a4paper]{article}

\usepackage{graphicx}
\usepackage{amsmath, amssymb, amsfonts, amsthm}
\usepackage{hyperref}
\usepackage{xspace}
\usepackage{color}
\usepackage{marginnote}
\usepackage{fancybox}

\def\R{{\mathbb R}}
\def\al{\alpha}
\def\be{\beta}
\def\ga{\gamma}

\def\eps{\varepsilon}
\def\la{\lambda}
\def\phi{\varphi}

\DeclareMathOperator{\e}{e}
\DeclareMathOperator{\im}{im}
\DeclareMathOperator{\sign}{sign}
\DeclareMathOperator{\cone}{cone}
\DeclareMathOperator{\conv}{conv}

\DeclareMathOperator{\supp}{supp}

\newcommand{\trans}{\mathsf{T}} %transpose
\newcommand{\mal}{\! \cdot}
\newcommand{\bd}{\partial}
\newcommand{\abs}[1]{\lvert #1 \rvert}
\newcommand{\dd}[2]{\frac{\mathrm{d} #1}{\mathrm{d} #2}}
\newcommand{\DD}[2]{\frac{\partial #1}{\partial #2}}
\newcommand{\zp}{{\oplus}}
\newcommand{\pp}{(+,\ldots,+)^\trans}

\newcommand{\ce}[1]{{\sf{#1}}}

\newcommand{\Fc}{F_c}
\newcommand{\ixl}[2]{I_{#1,#2}}
\newcommand{\Seq}{X}
\newcommand{\seq}{x}
\newcommand{\xa}{x^*} % alternative x in Minty's lemma
\newcommand{\pa}{\tilde w} % parameter in example
\newcommand{\iiis}{iv}

\newcommand{\nondeg}{non\-degenerate\xspace}
\newcommand{\degdeg}{degenerate\xspace}

\newtheorem{thm}{Theorem}%[section]
\newtheorem{pro}[thm]{Proposition}
\newtheorem{lem}[thm]{Lemma}
\newtheorem{cor}[thm]{Corollary}
\theoremstyle{definition}
\newtheorem{dfn}[thm]{Definition}
\newtheorem{exa}[thm]{Example}

\makeatletter
\def\blfootnote{\xdef\@thefnmark{}\@footnotetext}
\makeatother

\begin{document}

\title{%
On the bijectivity of families of exponential/generalized polynomial maps
}

\author{
Stefan~M\"uller, Josef~Hofbauer, Georg~Regensburger
}

% \date{Received: date / Accepted: date}

\maketitle

\begin{abstract}
\noindent
We start from a parametrized system of $d$ generalized polynomial equations (with real exponents)
for $d$ positive variables, involving $n$ generalized monomials with $n$ positive parameters.
Existence and uniqueness of a solution for all parameters and for all right-hand sides
is equivalent to the bijectivity of (every element of) a family of generalized polynomial/exponential maps.
We characterize the bijectivity of the family of exponential maps in terms of two linear subspaces arising from the coefficient and exponent matrices, respectively.
In particular, we obtain conditions in terms of sign vectors of the two subspaces
and a nondegeneracy condition involving the exponent subspace itself.
Thereby, all criteria can be checked effectively.
Moreover, we characterize when the existence of a unique solution is robust with respect to small perturbations of the exponents or/and the coefficients.
In particular, we obtain conditions in terms of sign vectors of the linear subspaces or, alternatively,
in terms of maximal minors of the coefficient and exponent matrices. 
Finally, we present applications to chemical reaction networks
with (generalized) mass-action kinetics. 

\vspace{1ex} \noindent

\vspace{1ex} \noindent
\end{abstract}

\blfootnote{
\scriptsize

\noindent
{Stefan~M\"uller} $\cdot$ Josef~Hofbauer \\
Faculty of Mathematics, University of Vienna, Oskar-Morgenstern-Platz 1, 1090 Wien, Austria \\[1ex]
{Georg Regensburger} \\
Institute for Algebra, Johannes Kepler University Linz, Altenberger Stra{\ss}e 69, 4040 Linz, Austria \\[1ex]
Corresponding author: 
\href{mailto:st.@univie.ac.at}{st.mueller@univie.ac.at}
}

%%%%%%%%% %%%%%%%%% %%%%%%%%% %%%%%%%%% %%%%%%%%%

\section{Introduction}

Given two matrices $W = (w^1, \ldots, w^n)$, $\tilde W = (\tilde w^1, \ldots, \tilde w^n) \in \R^{d \times n}$ with \parbox[t]{6ex}{$d \le n$} and full rank,
consider the parametrized system of generalized polynomial equations 
\[
\sum_{j = 1}^n w_{ij} \, c_j \, x_1^{\tilde w_{1j}} \cdots x_d^{\tilde w_{d\hspace{.3pt}j}} = y_i , \quad i=1, \ldots , d
\]
for $d$ positive variables $x_i>0$ (and right-hand sides $y_i$),
involving the `monomials' $c_j \, x_1^{\tilde w_{1j}} \cdots x_d^{\tilde w_{d\hspace{.3pt}j}} = c_j \, x^{\tilde w^j}$, $j=1, \ldots , n$,
in particular,
the $n$ positive parameters~$c_j>0$.
In other words, $x \in \R^d_{>0}$, $y \in \R^d$, and $c \in \R^n_{>0}$.
As in the theory of fewnomials~\cite{Khovanskii1991,Sotille2011},
the monomials are given, however, with a positive parameter associated to every monomial.

Writing the vector of monomials as $c \circ x^{\tilde W} \in \R^n_{>0}$,
thereby introducing $x^{\tilde W} \in \R^n_{>0}$ as $(x^{\tilde W})_j = x^{\tilde w^j}$
and denoting componentwise multiplication by $\circ$,
yields the compact form
\[
W (c \circ x^{\tilde W}) = y .
\]
Note that, for the existence of a positive solution $x$, the right-hand side $y$ must lie in the interior of $C = \cone W$,
the polyhedral cone generated by the columns of $W$.
The question arises
whether the above equation system has a unique positive solution $x \in \R^d_{>0}$,
for all right-hand sides $y \in C^\circ \subseteq \R^d$ and all positive parameters $c \in \R^n_{>0}$.
This question is equivalent to
whether the generalized polynomial map $f_c \colon \R^d_{>0} \to C^\circ \subseteq \R^d$,
\[
f_c(x) = W (c \circ x^{\tilde W})
\]
or, equivalently, the exponential map $\Fc \colon \R^d \to C^\circ \subseteq \R^d$,
\[
\Fc(x) = W (c \circ \e^{\tilde W^\trans x})
\] 
is bijective for all $c \in \R^n_{>0}$.

In the context of chemical reaction networks (CRNs) with {\em generalized} mass-action kinetics \cite{mueller:regensburger:2012,mueller:regensburger:2014},
the question is equivalent to whether every set of complex-balanced equilibria (an `exponential manifold')
intersects every stoichiometric class (an affine subspace) in exactly one point.
For a motivation from CRNs, see Section~\ref{sec:appl} or~\cite{CMPY2018}.
The assumption of mass-action kinetics corresponds to $W=\tilde W$,
and in this case
there is indeed exactly one complex-balanced equilibrium in every stoichiometric class. 

In case $W = \tilde{W}$,
the map $\Fc$ also appears in toric geometry \cite{Fulton1993}, where it is related to moment maps,
and in statistics \cite{PachterSturmfels2005}, where it is related to log-linear models.
The following result (called Birch's Theorem in~\cite{Sturmfels2002,PachterSturmfels2005,CraciunDickensteinShiuSturmfels2009,CraciunGarcia-PuenteSottile2010,GopalkrishnanMillerShiu2014,CMPY2018}) guarantees the bijectivity of $\Fc$ for all $c > 0$. 

\begin{thm}[\cite{Fulton1993}, Section~4.2] \label{fulton} 
Let $W=\tilde{W}$.
Then the map $\Fc$ is a real analytic isomorphism of $\R^d$ onto $C^\circ$ for all $c>0$.
\end{thm}

In this work, we characterize the simultaneous bijectivity of the maps $\Fc$ for all $c>0$ (for given coefficients $W$ and exponents $\tilde W$)
in terms of (sign vectors of) the linear subspaces $S = \ker W \subseteq \R^n$ and $\tilde S = \ker \tilde W \subseteq \R^n$, see Theorem~\ref{bij}.
Moreover, we characterize the robustness of bijectivity with respect to small perturbations 
of the exponents $\tilde W$ or/and the coefficients $W$,
corresponding to small perturbations of the subspaces $\tilde S$ and $S$ (in the Grassmannian),
see Theorems~\ref{pert_tildeS}, \ref{pert_S}, \ref{pert_S_tildeS}.

Sufficient conditions for bijectivity have been given in previous work~\cite{mueller:regensburger:2012},
using Brouwer degree, and parallel work~\cite{CMPY2018}, using differential topology.
For a smaller class of maps~\cite{G2009}, bijectivity has been proved, using Brouwer's fixed point theorem.

Our main technical tool is Hadamard's global inversion theorem
which essentially states that a $C^1$-map is a diffeomorphism
if and only if it is locally invertible and proper.
By previous results~\cite{CraciunGarcia-PuenteSottile2010,mueller:regensburger:2012},
the map $\Fc$ is locally invertible for all $c>0$ if and only if it is injective for all $c>0$
which can be characterized 
in terms of maximal minors of $W$ and $\tilde W$ or, equivalently,
in terms of sign vectors of the subspaces $S$ and $\tilde S$,
see Subsection~\ref{sec:inj}.
Most importantly, we show that $\Fc$ is proper if and only if it is `proper along rays'
and that properness for all $c>0$ can be characterized in terms of sign vectors of $S$ and $\tilde S$, 
together with a nondegeneracy condition depending on the subspace $\tilde S$ itself.

The crucial role of sign vectors in the characterization of existence and uniqueness of positive solutions to parametrized polynomial equations
suggests a comparison with Descartes' rule of signs for univariate (generalized) polynomials~\cite{Struik1969,Laguerre1883,Jameson2006}.
%Consider a univariate polynomial and order the monomials by their exponents. 
%Now, let $s$ be the number of sign changes in the sequence of (nonzero) coefficients,
%and let $p$ be the number of positive roots (where multiple roots are counted separately).
%Then, Descartes' rule~\cite{Struik1969} states that $p \le s$ and $s-p$ is even. 
%As shown by Laguerre~\cite{Laguerre1883,Jameson2006} the same statement holds for generalized monomials (with real exponents).
%More recently it has been shown that the upper bound is sharp~\cite{AlbouyFu2014}:
%for given sign sequence, there exist coefficients such that $p=s$.
%Hence 
A sharp rule~\cite{AlbouyFu2014} states that 
a univariate polynomial with given sign sequence has exactly one positive solution for all (positive) coefficients 
if and only if there is exactly one sign change.
Indeed, this statement follows from our main result % for univariate polynomials,
% The proof of equivalence involves a case distinction on which monomial plays the role of the right-hand side.
which can be seen as a multivariate generalization of the sharp Descartes' rule for exactly one positive solution.

\subsection*{Organization of the work and main results}

In Section~\ref{sec:fam}, we introduce the family of exponential maps $\Fc$ with $c>0$
and discuss previous results on injectivity.

In Section~\ref{sec:bij}, we present our main result, Theorem~\ref{bij}, 
characterizing the simultaneous bijectivity of the maps $\Fc$,
and the crucial Lemmas~\ref{seq} and \ref{ray},
regarding the properness of $\Fc$.
In Subsection~\ref{sec:special}, we discuss two extreme cases
regarding the geometry of the cone $C$,
% the polyhedral cone generated by the columns of~$W$;
namely, $C=\R^d$ and $C$ is pointed.
In Subsection~\ref{sec:sv}, we show that the simultaneous bijectivity of the maps $\Fc$ cannot be characterized in terms of sign vectors only, cf.~Example~\ref{exa:sv}.
Still, there are sufficient conditions for bijectivity in terms of sign vectors or in terms of faces of the Newton polytope,
cf.~Propositions~\ref{iii} and~\ref{newton}.

In Section~\ref{sec:robust}, we study the robustness of simultaneous bijectivity.
In Subsection~\ref{sec:exp}, we consider perturbations of the exponents $\tilde W$
and show that robustness of bijectivity is equivalent to robustness of injectivity
which can be characterized in terms of sign vectors, cf.~Theorem~\ref{pert_tildeS}.
The criterion involves the closure of a set of sign vectors
and represents a simple sufficient condition for bijectivity, cf.~Proposition~\ref{cc_bij}.
Equivalently, robustness can be characterized in terms of maximal minors.
% Lemma~\ref{pert_inj_cc}, Proposition~\ref{pert_inj=cc}, Proposition~\ref{cc_bij}, Example~\ref{exa:cc}, and Theorem~\ref{pert_tildeS}. cf.~Proposition~\ref{cc_det}.
In Subsection~\ref{sec:coeff}, we consider perturbations of the coefficients $W$
and characterize robustness of bijectivity again in terms of sign vectors (including another closure condition), cf.~Theorem~\ref{pert_S}.
In particular, robustness of bijectivity implies that either $C=\R^d$ or $C$ is pointed.
% In the latter case, the faces of $C$ are robustly generated.
% Proposition~\ref{cc'_i_iii}, Lemmas \ref{ii} and \ref{ii_pointed}, and Theorem~\ref{pert_S}.
Finally, in Subsection~\ref{sec:general}, 
we consider general perturbations (of both exponents and coefficients)
and characterize robustness of bijectivity
in terms of sign vectors and maximal minors, cf.~Theorem~\ref{pert_S_tildeS}.

In Section~\ref{sec:appl},
we present a derivation of our main problem from chemical reaction networks
and applications of our main results.
In particular, 
we formulate a deficiency zero theorem for generalized mass-action kinetics
and a robust deficiency zero theorem for (generalized) mass-action kinetics,
%the robustness of the deficiency zero theorem for mass-action kinetics
%with respect to small perturbations of the kinetic orders (from the stoichiometric coefficients)
cf.~Theorems~\ref{thm:generalized} and \ref{thm:generalizedrobust}.

Finally,
we provide appendices on % (A) a motivation from chemical reaction networks,
\eqref{sv} oriented matroids and \eqref{ta} a theorem of the alternative.

\subsection*{Notation}

We denote the positive real numbers by $\R_{>0}$ and the nonnegative real numbers by $\R_{\ge0}$. 
We write $x>0$ for $x \in \R^n_{>0}$ and $x \ge 0$ for $x \in \R^n_{\ge0}$.
For vectors $x,y \in \R^n$, we denote their scalar product by $x \cdot y$ and their componentwise (Hadamard) product by $x \circ y$.

For a vector $x \in \R^n$,
we obtain the sign vector $\sign(x) \in \{-,0,+\}^n$ by applying the sign function componentwise,
and we write
\[
\sign(S) = \{ \sign(x) \mid x \in S \}
\]
for a subset $S \subseteq \R^n$.

For a vector $x \in F^n$ with $F=\R$ or $F=\{-,0,+\}$, we denote its support by $\supp(x) = \{ i \mid x_i \neq 0 \}$. 
For a subset $X \subseteq F^n$,
we say that a nonzero vector $x \in X$ has (inclusion-)minimal support,
if $\supp(x') \subseteq \supp(x)$ implies $\supp(x') = \supp(x)$ for all nonzero $x' \in X$.

For a sign vector $\tau \in \{-,0,+\}^n$, we introduce 
\[
\tau^- = \{ i \mid \tau_i = -\}, \quad \tau^0 = \{ i \mid \tau_i = 0\}, \quad \text{and} \quad \tau^+ = \{ i \mid \tau_i = +\} .
\]
In particular, $\supp(\tau) = \tau^- \cup \tau^+$.
For a subset $T \subseteq \{-,0,+\}^n$, we write
\[
T_\zp = T \cap \{0,+\}^n . 
\]

The inequalities $0<-$ and $0<+$ induce a partial order on $\{-,0,+\}^n$:
for sign vectors $\tau, \rho \in \{-,0,+\}^n$, we write $\tau \le \rho$ if the inequality holds componentwise.
% For $x,y \in \R^n$, we say that $x$ {\em conforms to} $y$, if $\sign(x) \le \sign(y)$.
% 
The product on $\{-,0,+\}$ is defined in the obvious way.
For $\tau,\rho \in \{-,0,+\}^n$, we write $\tau \cdot \rho = 0$ ($\tau$ and $\rho$ are orthogonal)
if either $\tau_i \rho_i = 0$ for all $i$ or there exist $i,j$ with $\tau_i \rho_i = -$ and $\tau_j \rho_j = +$.
%Note that $\tau \perp \rho$ if and only if there are orthogonal vectors $x,y \in \R^n$ such that $\sign(x) = \tau$ and $\sign(y) = \rho$.
%
For $T \subseteq \{-,0,+\}^n$, we introduce the orthogonal complement
\[
T^\perp = \{ \tau \in \{-,0,+\}^n \mid \tau \cdot \rho = 0 \text{ for all } \rho \in T \} \, .
\]
% In particular, for a subspace $S$ of $\R^n$, $\sign(S^\perp)=\sign(S)^\perp$, cf.~\cite[Prop.~6.8]{Ziegler1995}.
%
Moreover, for $\tau,\rho \in \{-,0,+\}^n$,
we define the composition $\tau \circ \rho \in \{-,0,+\}^n$ as $(\tau \circ \rho)_i = \tau_i$ if $\tau_i \neq 0$ and  $(\tau \circ \rho)_i = \rho_i$ otherwise.

For a matrix $W \in \R^{d \times n}$, we denote its column vectors by $w^1, \ldots, w^n \in \R^d$.
For any natural number $n$, we define $[n] = \{1, \ldots , n\}$.
For $W \in \R^{d \times n}$ with $d \le n$ and $I \subseteq [n]$ of cardinality $d$,
we denote the square submatrix of $W$ with column indices in $I$ by $W_I$.
% and the corresponding maximal minor by $w_I = \det(W_I)$.

%%%%%%%%% %%%%%%%%% %%%%%%%%% %%%%%%%%% %%%%%%%%%

\section{Families of exponential maps} \label{sec:fam}

Let $W \in \R^{d \times n}$, $\tilde{W} \in \R^{\tilde{d} \times n}$ be matrices with $d, \tilde{d} \le n$ and full rank.
Further, let
\[
C = \cone W \subseteq \R^d % \quad \text{and} \quad \tilde C = \cone \tilde W \subseteq \R^{\tilde d}
\]
be the cone generated by the columns of $W$. % and $\tilde W$, respectively,
%and
%\[
%tilde P = \conv \tilde W \subseteq \R^{\tilde d}
%\]
%be the associated (Newton) polytope. 
%Since $W$ and $\tilde W$ have full rank, the cones $C$ and $\tilde C$ have nonempty interior, and we write $C^\circ = \inn(C)$.
Since $W$ has full rank, the cone $C$ has nonempty interior $C^\circ$.
Finally,
let $c>0$.
We define the exponential map
\begin{equation}
\begin{aligned}  \label{map}
\Fc \colon & \R^{\tilde{d}} \to C^\circ \subseteq \R^d \\
& x \mapsto 
W (c \circ \e^{\tilde W^\trans x} ) 
= \sum_{i=1}^n c_i \e^{\tilde{w}^i \mal x} w^i
\end{aligned}
\end{equation}
and the related subspaces
\begin{equation} \label{SS}
S = \ker W \subseteq \R^n \quad \text{and} \quad \tilde S = \ker \tilde W \subseteq \R^n .
\end{equation}

Note that injectivity and surjectivity of $\Fc$ only depend on $S$ and $\tilde S$.
In fact, let $V \in \R^{d \times n}$, $\tilde V \in \R^{\tilde{d} \times n}$ be such that $\ker V = S$, $\ker \tilde V = \tilde S$,
and let 
\[
G_c(x) = V (c \circ \e^{\tilde V^\trans x} ) 
\]
be the corresponding exponential map.
Then $V = U W$, $\tilde V = \tilde U \tilde W$ for invertible matrices $U \in \R^{d \times d}$, $\tilde U \in \R^{\tilde d \times \tilde d}$,
and
\[
G_c(x) = U \Fc(\tilde U^\trans x) .
\]

%%%%%%%%% %%%%%%%%% %%%%%%%%% %%%%%%%%% %%%%%%%%%

\subsection{Previous results on injectivity} \label{sec:inj}

In the context of multiple equilibria in mass-action systems \cite{CraciunFeinberg2005}
and geometric modeling \cite{CraciunGarcia-PuenteSottile2010},
where $d = \tilde d$,
it was shown that the map $\Fc$ is injective for all $c>0$
if and only if
$\Fc$ is a local diffeomorphism for all $c>0$.

\begin{thm}[Theorem~7 and Corollary 8 in \cite{CraciunGarcia-PuenteSottile2010}] \label{crac}
Let $\Fc$ be as in~\eqref{map} with $d = \tilde d$.
Then the following statements are equivalent:
\begin{enumerate}
\item
$\Fc$ is injective for all $c>0$.
\item
$\det(\DD{\Fc}{x}) \neq 0$ for all $x$ and all $c>0$.
\item
$\det(W_I) \det(\tilde W_I) \ge 0$ for all subsets $I \subseteq [n]$ of cardinality $d$ (or `$\le 0$' for all $I$)
and $\det(W_I) \det(\tilde W_I) \neq 0$ for some $I$.
\end{enumerate}
\end{thm}

In \cite{mueller:regensburger:2012}, we gave an alternative proof of this result
and extended it to the case $d\not=\tilde{d}$,
by using the sign vectors of the subspaces $S$ and $\tilde S$.

\begin{thm}[Theorem~3.6 in \cite{mueller:regensburger:2012}] \label{inj}
Let $\Fc$ be as in~\eqref{map} and $S,\tilde S$ be as in~\eqref{SS}.
Then the following statements are equivalent:
\begin{enumerate}
\item
$\Fc$ is injective for all $c>0$.
\item
$\Fc$ is an immersion for all $c>0$. \\
($\DD{\Fc}{x}$ is injective for all $x$ and all $c>0$.)
\item
$\sign(S) \cap \sign(\tilde S^\perp) = \{0\}$.
\end{enumerate}
\end{thm}

Theorems~\ref{crac} and~\ref{inj} characterize the simultaneous injectivity of $\Fc$ (with $d = \tilde d$) for all $c > 0$
equivalently in terms of maximal minors and sign vectors.

\begin{cor} \label{inj_det}
Let $S,\tilde S$ be subspaces of $\R^n$ of dimension $n-d$ (with $d \le n$).
For every $W, \tilde W \in \R^{d \times n}$ (with full rank $d$) such that $S = \ker W$ and $\tilde S = \ker \tilde W$,
the following statements are equivalent.
\begin{enumerate}
\item
$\sign(S) \cap \sign(\tilde S^\perp) = \{0\}$.
\item
$\det(W_I) \det(\tilde W_I) \ge 0$ for all subsets $I \subseteq [n]$ of cardinality $d$ (or `$\le 0$' for all $I$)
and $\det(W_I) \det(\tilde W_I) \neq 0$ for some $I$.
\end{enumerate}
\end{cor}

In the language of oriented matroids, Corollary~\ref{inj_det} relates {\em chirotopes} (signs of maximal minors of $W$ and $\tilde W$)
to {\em vectors} (sign vectors of $S = \ker W$ and $\tilde S = \ker \tilde W$),
see also Appendix~\ref{sv}.
Thereby, the sign vector condition is symmetric with respect to $S$ and $\tilde S$. % (that is, $W$ and $\tilde{W}$).

\begin{cor}[Corollary~3.8 in \cite{mueller:regensburger:2012}] \label{inj_cor}
Let $S,\tilde S$ be subspaces of $\R^n$ of equal dimension.
Then
\[
\sign(S) \cap \sign(\tilde S^\perp) = \{0\}
\quad \text{if and only if} \quad
\sign(\tilde S) \cap \sign(S^\perp) = \{0\} .
\]
\end{cor}

For a direct proof of Corollaries~\ref{inj_det} and~\ref{inj_cor},
see also \cite{Chaiken1996}.

In further works on injectivity of families of exponential/generalized polynomial maps,
the coefficient and exponent matrices need not have full rank,
and injectivity is studied on affine subspaces,
see~\cite{G2012,FW2012,MFR2016,BP2016}.

%%%%%%%%% %%%%%%%%% %%%%%%%%% %%%%%%%%% %%%%%%%%%
%%%%%%%%% %%%%%%%%% %%%%%%%%% %%%%%%%%% %%%%%%%%%
%%%%%%%%% %%%%%%%%% %%%%%%%%% %%%%%%%%% %%%%%%%%%

\section{Bijectivity} \label{sec:bij}

A necessary condition for the bijectivity of the map $\Fc$ is $d = \tilde d$.
In the rest of the paper, we consider $\Fc$ as in~\eqref{map} with $d = \tilde d$ and the related subspaces $S,\tilde S$ as in~\eqref{SS}.

A first sufficient condition for the bijectivity of the map $\Fc$ for all $c>0$ (in terms of sign vectors of $S$ and $\tilde S$)
was given in~\cite{mueller:regensburger:2012},
thereby extending Theorem~\ref{fulton}. % (Birch's Theorem).

\begin{thm}[Proposition~3.9 in~\cite{mueller:regensburger:2012}] \label{bij_suff}
% Let $\Fc$ be as in~\eqref{map} and $S,\tilde S$ as in~\eqref{SS}.
If $\sign(S) = \sign(\tilde S)$ and $\pp \in \sign(S^\perp)$,
then the map $\Fc$ is a real analytic isomorphism for all $c > 0$.
\end{thm}

As it will turn out, $\sign(S) = \sign(\tilde S)$ is sufficient for bijectivity, 
and the technical condition $\pp \in \sign(S^\perp)$ in \cite{mueller:regensburger:2012} is not needed, cf.~Corollary~\ref{cor:bij}.
We note that Theorems~\ref{crac}, \ref{inj}, and \ref{bij_suff} allowed a first multivariate generalization of Descartes' rule of signs 
for at most/exactly one positive solution,
see~\cite{MFR2016}.

In order to characterize the simultaneous bijectivity of the map $\Fc$ for all $c>0$,
we start with the following observation.

\begin{pro}
The following statements are equivalent.
\begin{enumerate}
\item
$\Fc$ is bijective for all $c>0$.
\item
$\Fc$ is a diffeomorphism for all $c>0$.
\item
$\Fc$ is a real analytic isomorphism for all $c>0$.
\end{enumerate}
\end{pro}
\begin{proof}
Let $\Fc$ be bijective for all $c>0$.
In particular, it is injective, and $\det(\DD{\Fc}{x}) \neq 0$ for all $x$ and $c>0$,
by Theorems~\ref{crac} or \ref{inj}.
Hence, $\Fc$ is a local diffeomorphism for all $c>0$.
%and hence a diffeomorphism for all $c>0$.
Further, $\Fc$ is real analytic and hence a local real analytic isomorphism for all $c>0$.
\end{proof}

Most importantly, we will use Hadamard's global inversion theorem.

\begin{thm}[\cite{Hadamard1906}, Theorem A in~\cite{Gordon1972}]
A $C^1$-map $F \colon \R^d \to \R^d$ is a diffeomorphism
if and only if
the Jacobian $\det(\DD{F}{x}) \neq 0$ for all $x \in \R^d$ and $\abs{F(x)} \to \infty$ whenever $\abs{x} \to \infty$.
\end{thm}

Obviously, we need a slightly more general version of this result 
which follows from Satz II in \cite{BanachMazur1934} or Theorem B in~\cite{Gordon1972}.
\begin{thm} \label{had}
Let $U \subseteq \R^d$ be open and convex. %(or simply connected)
A $C^1$-map $F \colon \R^d \to U$ is a diffeomorphism
if and only if
the Jacobian $\det(\DD{F}{x}) \neq 0$ for all $x \in \R^d$ and $F$ is proper.
\end{thm}

Recall that a map $F$ between two topological spaces is {\em proper}, if  $F^{-1}(K)$ is compact for each compact subset $K$ of the target space.
This is obviously necessary for the inverse $F^{-1}$ to be continuous.

\begin{lem} \label{had_lem}
Let $U \subseteq \R^d$ be open.
A continuous map $F \colon \R^d \to U$ is proper if and only if,
for sequences $\seq_n$ in $\R^d$ with $\abs{\seq_n} = 1$ and $\seq_n \to \seq$ and $t_n$ in $\R_{>0}$ with $t_n \to \infty$,
$F(\seq_n t_n) \to y$ implies $y \in \bd U$. % $y \notin U$ (and hence $y \in \bd U$).
\end{lem}
\begin{proof}
Suppose $F$ is proper and $F(\seq_n t_n) \to y$, but $y \in U$.
Take a closed ball $K \subseteq U$ around $y$. % $K = K_{\eps}[y] \subseteq U$.
Then $F^{-1}(K)$ contains the unbounded sequence $(\seq_n t_n)_{n \ge N}$ for some positive $N$ 
%so $F$ cannot be proper.
and hence is not compact, a contradiction.

Conversely,
let $K$ be a compact subset of $U$. We need to show that every sequence $\Seq_n$ in $F^{-1}(K)$ % (i.e., $F(x_n) \in K$)
has an accumulation point. Since $F^{-1}(K)$ is closed, we only need to show that $\Seq_n$ has a bounded subsequence.
Suppose not, then $\abs{\Seq_n} \to \infty$. Since $F(\Seq_n) \in K$, there is a subsequence (call it $\Seq_n$ again) such that $F(\Seq_n) \to y \in K$.
Now there is a subsubsequence (call it $\Seq_n$ again) such that $\seq_n = \Seq_n / \abs{\Seq_n} \to \seq$,
that is, the sequence $\seq_n$ on the unit sphere converges.
With $t_n = \abs{\Seq_n}$, we have $F(\seq_n t_n) \to y \in K \subset U$, a contradiction.
\end{proof}

% If $F$ is proper, then either $\abs{F(\seq_n t_n)} \to \infty$ as $n \to \infty$ or $F(\seq_n t_n) \to y \notin U$.
In particular, if $F$ is proper, then, for all nonzero $x \in \R^d$,
$F(xt) \to y$ as $t \to \infty$ implies $y \in \bd U$.
That is, if the function values converge along a ray,
then the limit lies on the boundary of the range. % (and hence on its boundary).

By Lemma~\ref{seq} below, the map $\Fc$ under consideration is proper, if it is `proper along rays'.
Before we prove this result, we discuss the behaviour of $\Fc$ along a ray.
For $x \in \R^d$ and $\la \in \R$, we introduce
\[
\ixl{x}{\la} = \{ i \mid \tilde w^i \mal x = \la \}
\]
and write
\[
\Fc(xt) = \sum_\la \sum_{i \in \ixl{x}{\la}} c_i \e^{\la t} w^i ,
\]
where a sum over the empty set is defined as zero.
For $x \in \R^d$ and $c>0$,
let $\la_{\max}$ be the largest $\la$ such that $\sum_{i \in \ixl{x}{\la}} c_i w^i \neq 0$.
If $\la_{\max}>0$, then 
\[
\Fc(xt) \e^{-\la_{\max} t} \to \sum_{i \in \ixl{x}{\la_{\max}}} c_i w^i \neq 0
\]
as $t \to \infty$
and hence $\abs{\Fc(xt)} \to \infty$.
If $\la_{\max}\le0$, then
\[
\Fc(xt) \to \sum_{i \in \ixl{x}{0}} c_i w^i \in C 
\]
as $t \to \infty$.
In this case, 
any vector $w^i$ with $i \in \ixl{\seq}{\la}$ and $\la>0$ (and hence $\sum_{i \in \ixl{x}{\la}} c_i w^i = 0$) lies in the lineality space of~$C$, see~Appendix~\ref{sv}.
If $\la_{\max}<0$, %(and hence $\ixl{x}{0} = \emptyset$), 
then $\Fc(xt) \to 0$.
As a result, we have the following fact:
\[
\text{For every } x \in \R^d \text{, either } \abs{\Fc(xt)} \to \infty \text{ as } t \to \infty \text{ or }\Fc(xt) \to y \in C .
\]

\begin{lem} \label{seq}
The map $\Fc$ is proper, if
\begin{equation} \label{rc} \tag{\textasteriskcentered}
\Fc(xt) \to y \quad \text{as} \quad t \to \infty \quad \text{implies} \quad y \in \bd C
\end{equation}
for all nonzero $x \in \R^d$. % then $\Fc$ is proper.
\end{lem}
\begin{proof}
We assume that the ray condition~\eqref{rc} holds for all nonzero $x \in \R^d$.

Let $\seq \in \R^d$ with $\abs{\seq}=1$. %and let $c>0$. 
In order to apply Lemma~\ref{had_lem},
we consider sequences $\seq_n$ in $\R^d$ with $\abs{\seq_n} = 1$ and $\seq_n \to \seq$ and $t_n$ in $\R_{>0}$ with $t_n \to \infty$.
%and assume $\Fc(\seq_n t_n) \to y'$.
%Ultimately, we have to show $y' \in \bd C$. 

To begin with, we show that $\abs{\Fc(\seq t)} \to \infty$ as $t \to \infty$ implies $\abs{\Fc(\seq_n t_n)} \to \infty$ as $n \to \infty$.
Suppose $\abs{\Fc(\seq t)} \to \infty$,
that is, there is $\la>0$ such that $\Fc(\seq t) \e^{-\la t} \to \sum_{i \in \ixl{\seq}{\la}} c_i w^i \neq 0$ as $t \to \infty$.
For $\seq'$ close to $\seq$, we have the partition \[ \ixl{\seq}{\la} = \ixl{\seq'}{\mu_1} \cup \cdots \cup \ixl{\seq'}{\mu_p} \]
with $\mu_j$ close to $\la$ and hence $\mu_j>\frac{\la}{2}$.
Hence, there exists a largest $\mu_j$ such that $\sum_{i \in \ixl{\seq'}{\mu_j}} c_i w^i \neq 0$. 
Otherwise, \[ \sum_{i \in \ixl{\seq}{\la}} c_i w^i = \sum_{i \in \ixl{\seq'}{\mu_1}} c_i w^i + \ldots + \sum_{i \in \ixl{\seq'}{\mu_p}} c_i w^i = 0 . \]
Additionally, there may exist an even larger $\mu$ with $\sum_{i \in \ixl{\seq'}{\mu}} c_i w^i \neq 0$.
In any case,
there is $\la' > \frac{\la}{2}$ such that
\[
\Fc(\seq' t) \e^{-\la' t} \to \sum_{i \in \ixl{\seq'}{\la'}} c_i w^i \neq 0
\]
as $t \to \infty$
and hence $\abs{\Fc(\seq' t)} \e^{-\frac{\la}{2} t} > \ga$ with $\ga>0$ independent of $\seq'$;
that is,
$\abs{\Fc(\seq' t)} > \ga \e^{\frac{\la}2 t}$ as $t \to \infty$.
Hence 
$\abs{\Fc(\seq_n t_n)} >  \ga \e^{\frac{\la}2 t_n}$ as $n \to \infty$; that is, $\abs{\Fc(\seq_n t_n)} \to \infty$,
as claimed.
% contradicting $\Fc(\seq_n t_n) \to y'$. Hence, $\Fc(\seq t) \to y$ as $t \to \infty$.

In case $C=\R^d$ ($\bd C = \emptyset$), the ray condition~\eqref{rc} implies $\abs{\Fc(\seq t)} \to \infty$ as $t \to \infty$ and hence $\abs{\Fc(\seq_n t_n)} \to \infty$ as $n \to \infty$.
By Lemma~\ref{had_lem}, $\Fc$ is proper.

In case $C \neq \R^d$, assume $\Fc(\seq_n t_n) \to y'$ as $n \to \infty$. 
Then, $\Fc(\seq t) \to y$ as $t \to \infty$, by the first argument in the proof and the fact before the lemma.
In particular, $\sum_{i \in \ixl{\seq}{\la}} c_i w^i = 0$ for $\la > 0$
and $y = \sum_{i \in \ixl{\seq}{0}} c_i w^i$. % with $y \in \bd C$.
Hence, vectors $w^i$ with $i \in \ixl{\seq}{\la}$ and $\la>0$ lie in the lineality space of~$C$.
% and \[ \cone(w^i \mid i \in \ixl{\seq}{\la} \text{ with } \la>0) \subseteq \bd C . \]
By the ray condition~\eqref{rc}, $y \in \bd C$, and hence 
\[ 
\cone(w^i \mid i \in \ixl{\seq}{0}) \subseteq \bd C . 
\]
% Note that, if $y = 0$, then the vectors $w^i$ with $i \in \ixl{\seq}{0}$ lie in the lineality space of~$C$ and hence in $\bd C$.
Finally, we write
\[
\Fc(\seq_n t_n) = \sum_{i=1}^n c_i \e^{\tilde w^i \mal \seq_n \, t_n} w^i 
= \sum_{\la} \sum_{i \in \ixl{\seq}{\la}} c_i \e^{\tilde w^i \mal \seq_n \, t_n} w^i .
\]
%where $\sum_\la$ denotes a finite sum over all $\la \in \R$ with $\ixl{x}{\la}\neq\emptyset$.
For $\seq_n$ close to $\seq$, we have $\tilde w^i \mal \seq_n$ close to $\la$ for $i \in \ixl{x}{\la}$,
in particular, $\sum_{i \in \ixl{\seq}{\la}} c_i \e^{\tilde w^i \mal \seq_n \, t_n} w^i \to 0$ for $\la<0$.
The limit $\Fc(\seq_n t_n) \to y'$ as $n \to \infty$ implies
\[
\sum_{i \in \ixl{\seq}{0}} c_i \e^{\tilde w^i \mal \seq_n \, t_n} w^i
+ \sum_{\la > 0} \sum_{i \in \ixl{\seq}{\la}} c_i \e^{\tilde w^i \mal \seq_n \, t_n} w^i 
\to y' ,
\]
and  $y' \in \bd C$ 
since the sum of a vector in $\bd C$ and a vector in the lineality space of $C$ lies in $\bd C$.
By Lemma~\ref{had_lem}, $\Fc$ is proper.
\end{proof}

Let $\Fc(xt)\to y$ as $t\to\infty$ along the ray given by $x$
and $\Fc(x_nt_n)\to y'$ as $n\to\infty$ for a sequence $x_nt_n$ (with $x_n\to x$ and $t_n\to\infty$), approaching the ray.
In the proof of Lemma~\ref{seq}, we have shown that, 
if $y=0$, then $y'\in L$, where $L$ is the lineality space of $C$.
In general, if $y \in C_x = \cone(w^i \mid i \in \ixl{\seq}{0})$, then $y' \in C_x+L$.
Note that there are only finitely many index sets $\ixl{x}{0}$ and hence finitely many limit points $y=\sum_{i \in \ixl{\seq}{0}} c_i w^i$ (for fixed $c>0$),
whereas every $y' \in \bd C$ arises as a limit point (if $\Fc$ is surjective).

Using Theorem~\ref{had} (Hadamard's global inversion theorem) together with Theorems~\ref{crac} or \ref{inj} and Lemma~\ref{seq}, we summarize our findings.
\begin{cor} \label{cor:before}
The map $\Fc$ is bijective for all $c>0$
if and only if
$\Fc$ is injective for all $c>0$ and the ray condition~\eqref{rc} in Lemma~\ref{seq} holds for all nonzero $x \in \R^d$ and all $c>0$.
\end{cor}

By Theorems~\ref{crac} or~\ref{inj}, 
the simultaneous injectivity of $\Fc$ for all $c>0$ can be characterized in terms of sign vectors of the subspaces $S$ and $\tilde S$.
By Lemma~\ref{ray} below, 
the ray condition~\eqref{rc} (for all nonzero $x \in \R^d$ and all $c>0$) can be characterized in terms of sign vectors of $S$ and $\tilde S$
together with a nondegeneracy condition depending on sign vectors of $S$ and on the subspace $\tilde S$ itself.

\begin{dfn} \label{def}
Let $S,\tilde S$ be subspaces of $\R^n$.
% A vector $z \in \tilde S^\perp$ with a positive component is called {\em \nondeg} if 
The pair $(S,\tilde S)$ is called {\em \nondeg} if, for every $z \in \tilde S^\perp$ with a positive component,
\begin{itemize}
\item
there is $I = \{ i \mid z_i = \la \}$ with $\la>0$,
defining $\pi \in \{0,+\}^n$ with $\pi^+ = I$,
such that $\pi \notin \sign(S)_\zp$
or
\item
for $\tilde \tau = \sign(z) \in \sign(\tilde S^\perp)$, there is a nonzero $\tau \in \sign(S^\perp)_\zp$ such that $\tilde \tau^0 \subseteq \tau^0$.
\end{itemize}
% The pair $(S,\tilde S)$ is called \nondeg if every $z \in \tilde S^\perp$ with a positive component is \nondeg.
\end{dfn}

As our main result,
we obtain a characterization of the simultaneous bijectivity of $\Fc$ for all $c>0$ in terms of the subspaces $S$ and $\tilde S$.

\begin{thm} \label{bij}
The map $\Fc$ is a diffeomorphism for all $c>0$ if and only if
\begin{itemize}
\item[(i)]
$\sign(S) \cap \sign(\tilde S^\perp) = \{0\}$,
\item[(ii)]
for every nonzero $\tilde \tau \in \sign(\tilde S^\perp)_\zp$, 
there is a nonzero $\tau \in \sign(S^\perp)_\zp$ such that $\tau \le \tilde \tau$, 
and
\item[(iii)]
the pair $(S,\tilde S)$ is \nondeg.
\end{itemize}
\end{thm}

Theorem~\ref{bij} immediately implies Theorems~\ref{fulton} and~\ref{bij_suff} (`Birch's Theorem' and its first extension).
\begin{cor} \label{cor:bij}
The map $\Fc$ is a diffeomorphism for all $c>0$ if $\sign(S) = \sign(\tilde S)$.
\end{cor}
\begin{proof}
By Corollary~\ref{orthogonal} in Appendix~\ref{ta}, $\sign(S^\perp)=\sign(S)^\perp$.
Hence, $\sign(S) = \sign(\tilde S)$ implies conditions~(i) and (ii) in Theorem~\ref{bij}.
Now, for $z \in \tilde S^\perp$ with a positive component $z_i = \la > 0$,
consider $\pi \in \{0,+\}^n$ with $\pi^+ = \{ i \mid z_i = \la \}$
and $\tilde \tau = \sign(z) \in \sign(\tilde S^\perp)$.
Obviously, $\pi \cdot \tilde \tau \neq 0$ and hence $\pi \not\in \sign(\tilde S)_\zp = \sign(S)_\zp$.
That is, $(S,\tilde S)$ is \nondeg, as required by condition~(iii).
\end{proof}

We note that condition~(i) in Theorem~\ref{bij} can also be characterized in terms of maximal minors of the matrices $W$ and $\tilde W$, cf.~Corollary~\ref{inj_det}.

Condition~(ii) can be reformulated using faces of the cones $C = \cone W$ and $\tilde C = \cone \tilde W$:
\begin{itemize}
\item[(ii)]
for every proper face $\tilde f$ of $\tilde C$ with $\tilde I = \{ i \mid \tilde w^i \in \tilde f \}$, 
there is a proper face $f$ of $C$ with $I = \{ i \mid w^i \in f \}$ such that $\tilde I \subseteq I$.
\end{itemize}
Indeed, 
a face $f$ of $C$ with $I = \{ i \mid w^i \in f \}$ corresponds to a supporting hyperplane with normal vector $x$
such that $w^i \mal x = 0$ for $i \in I$
and $w^i \mal x > 0$ otherwise (for $w^i$ lying on the positive side of the hyperplane).
Hence $f$ is characterized by the nonnegative sign vector $\tau = \sign(W^\trans x) \in \sign(S^\perp)_\zp$ with $\tau^0 = I$.
Analogously, a face $\tilde f$ of $\tilde C$ with $\tilde I = \{ i \mid \tilde w^i \in \tilde f \}$ is characterized by a nonnegative sign vector $\tilde \tau \in \sign(\tilde S^\perp)_\zp$ with $\tilde \tau^0 = \tilde I$.
Clearly, $\tilde I \subseteq I$ is equivalent to $\tau \le \tilde \tau$.
(For more details on sign vectors and face lattices, see Appendix~\ref{sv}.)

Condition (iii) concerns nondegeneracy. 
The second condition in Definition~\ref{def}, on sign vectors $\tilde \tau = \sign(z) \in \sign(\tilde S^\perp)$, 
corresponds to condition (ii), on nonnegative sign vectors $\tilde \tau \in \sign(\tilde S^\perp)_\zp$. 
The first condition on $z \in \tilde S^\perp$ can also be interpreted geometrically (in terms of the columns of $W,\tilde W$).
Note that $\tilde S^\perp = (\ker \tilde W)^\perp = \im \tilde W^\trans$ and $z = \tilde W^\trans x$ for some $x\in\R^d$.
Hence, the set 
\[
I = \{ i \mid z_i = \la \} = \{ i \mid \tilde w^i \mal x = \la \} = \ixl{x}{\la}
\]
with $\la>0$ indicates equal positive components $z_i$
or, geometrically, equal positive projections of columns $\tilde w^i$ (on $x$).
The corresponding columns $w^i$ must not be positively dependent,
as expressed by the condition $\pi \notin \sign(S)_\zp$ 
for the nonnegative sign vector $\pi \in \{0,+\}^n$ with $\pi^+ = I$.
%The set $I$ defines the nonnegative sign vector $\tau \in \{0,+\}^n$ by $\tau^+ = I$.
%Finally, the condition $\tau \notin \sign(S)_\zp$ requires that the corresponding columns $w^i$ are not positively dependent.

It remains to prove Lemma~\ref{ray}.

\begin{lem} \label{ray}
The ray condition~\eqref{rc} in Lemma~\ref{seq} holds for all nonzero $x \in \R^d$ and for all $c>0$
if and only if
conditions (ii) and (iii) in Theorem~\ref{bij} hold.
\end{lem}
\begin{proof}
For nonzero $x \in \R^d$, let $\la_x = \max_i \tilde w^i \mal x$.
We show the following two statements.
Condition (ii) is equivalent to: the ray condition~\eqref{rc} holds for all nonzero $x$ with $\la_x \le 0$ and all $c>0$.
Condition (iii) is equivalent to: the ray condition~\eqref{rc} holds for all nonzero $x$ with $\la_x > 0$ and all $c>0$.

(ii): If $\la_x\le 0$, then $\tilde \tau = \sign(-\tilde W^\trans x) \in \sign(\tilde S^\perp)_\zp$ defines a proper face of~$\tilde C$
and $\Fc(xt) \to \sum_{i \in \tilde \tau^0} c_i w^i$ as $t \to \infty$.
The ray condition~\eqref{rc} for all $c>0$ is equivalent to $\sum_{i \in \tilde \tau^0} c_i w^i \in \bd C$ for all $c>0$.
That is, there is a proper face of~$C$ characterized by a nonzero $\tau \in \sign(S^\perp)_\zp$ such that $\tilde \tau^0 \subseteq \tau^0$.
Equivalently, $\tau \le \tilde \tau$, that is, (ii) for $\tilde \tau$.

By varying over all nonzero $x \in \R^d$ with $\la_x\le0$, all nonzero $\tilde \tau \in \sign(\tilde S^\perp)_\zp$ are covered.

(iii): If $\la_x> 0$, then $z = \tilde W^\trans x \in \tilde S^\perp$ has a positive component.
Using the fact before Lemma~\ref{seq},
the ray condition~\eqref{rc} for all $c>0$ is equivalent to

for all $c>0$,
\begin{itemize}
\item[($\al$)]
either there is $\la>0$ such that $\Fc(xt) \e^{-\la t} \to \sum_{i \in \ixl{x}{\la}} c_i w^i \neq 0$ as $t \to \infty$
\item[($\be$)]
or $\Fc(xt) \to \sum_{i \in \ixl{x}{0}} c_i w^i \in \bd C$. 
\end{itemize}
This is further equivalent to
\begin{itemize}
\item[(a)]
there is $\la>0$ such that, for all $c>0$,
$\sum_{i \in \ixl{x}{\la}} c_i w^i \neq 0$
or
\item[(b)]
$\sum_{i \in \ixl{x}{0}} c_i w^i \in \bd C$
for all $c>0$.
\end{itemize}
To see this, 
note that the sets $\ixl{x}{\la}$ are disjoint
and the sums $\sum_{i \in \ixl{x}{\la}} c_i w^i$ involve different coefficients~$c_i$ for different $\la$. \\
($\Rightarrow$): Assume $\neg$(a), that is, for all $\la>0$,
there exists $c>0$ such that $\sum_{i \in \ixl{x}{\la}} c_i w^i = 0$.
Then, $\sum_{i \in \ixl{x}{0}} c_i w^i \in \bd C$ for all $c>0$, that is, (b). \\
($\Leftarrow$): Clearly, (a) implies ($\al$) for all $c>0$.
Finally, assume (b) and let $c>0$.
Then, either ($\al$) or, for all $\la>0$, $\sum_{i \in \ixl{x}{\la}} c_i w^i = 0$. 
In the latter case, $\Fc(xt) \to \sum_{i \in \ixl{x}{0}} c_i w^i $ with $\sum_{i \in \ixl{x}{0}} c_i w^i \in \bd C$,
that is, ($\be$).

Finally, (a) or (b) is equivalent to
\begin{itemize}
\item[$\bullet$]
there is $\ixl{x}{\la} = \{ i \mid z_i = \la \}$ with $\la>0$
such that $c \notin \ker W = S$ for all $c \ge 0$ with $\supp(c) = \ixl{x}{\la}$,
that is, there is $\pi \in \{0,+\}^n$ with $\pi^+ = \ixl{x}{\la}$
such that $\pi \notin \sign(S)_\zp$,
or
\item[$\bullet$]
for $\tilde \tau = \sign(z) \in \sign(\tilde S^\perp)$ and hence $\tilde \tau^0 = \ixl{x}{0}$,
there is a proper face of $C$, characterized by a nonzero $\tau \in \sign(S^\perp)_\zp$, 
such that $\tilde \tau^0 \subseteq \tau^0$,
\end{itemize}
that is, (iii) for $z$.

By varying over all nonzero $x \in \R^d$ with $\la_x>0$, all $z \in \tilde S^\perp$ with a positive component are covered.
\end{proof}

%%%%%%%%% %%%%%%%%% %%%%%%%%% %%%%%%%%% %%%%%%%%%

\subsection{Special cases: $C=\R^d$ or $C$ is pointed} \label{sec:special}

We discuss the conditions for bijectivity in Theorem~\ref{bij} for two extreme cases,
regarding the geometry of the cones $C = \cone W$ and $\tilde C = \cone \tilde W$.

If $C=\R^d$ (that is, $\sign(S^\perp)_\zp = \{0\}$), % that is, $\pp \in \sign(S)$, \\
then condition (ii) is equivalent to % $\sign(\tilde S^\perp)_\zp = \{0\}$, that is, 
$\tilde C=\R^d$.
Hence, if $C=\R^d$ and $\Fc$ is bijective for all $c>0$, then $\tilde C=\R^d$.
However, the converse does not hold.

\begin{exa}
Let $\Fc$ be given by the matrices
\[
\tilde W =
\begin{pmatrix}
1 & 0 &-1 \\
0 & 1 &-1
\end{pmatrix} 
\quad \text{and} \quad
W =
\begin{pmatrix}
1 & 0 &-1 \\
0 & 1 & 0
\end{pmatrix} .
\]
Then $\tilde C = \R^2$ and $\Fc$ is bijective for all $c>0$.
However, $C \neq \R^2$.
\end{exa}

If $\pp \in \sign(S^\perp)$ % that is, $\sign(S)_\zp = \{0\}$, \\
(that is, $C$ is pointed and no column of $W$ is zero),
then condition (iii) holds (since $\sign(S)_\zp = \{0\}$),
and conditions (i) and (ii) imply $\pp \in \sign(\tilde S^\perp)$ (by Proposition~\ref{pro:pointed} below).
Hence, if $\pp \in \sign(S^\perp)$ and $\Fc$ is bijective for all $c>0$, then $\pp \in \sign(\tilde S^\perp)$.
However, the converse does not hold.

\begin{exa}
Let $\Fc$ be given by the matrices
\[
\tilde W =
\begin{pmatrix}
1 & 1 & 0 \\
0 & 1 & 1
\end{pmatrix} 
\quad \text{and} \quad
W =
\begin{pmatrix}
1 & 0 &-1 \\
0 & 1 & 0
\end{pmatrix} .
\]
Then, $\tilde C = \R^2_{\ge0}$, $(+,+,+)^\trans \in \sign(\tilde S^\perp)$,
and $\Fc$ is bijective for all $c>0$.
However, $C = \R \times \R_{\ge0}$ and $(+,+,+)^\trans \not\in \sign(S^\perp)$.
\end{exa}

If $\pp \in \sign(S^\perp)$ (that is, $C$ is pointed and no column of $W$ is zero),
then conditions (i) and (ii) imply the surjectivity of $\Fc$ for all $c>0$ and, 
by the following result, $\pp \in \sign(\tilde S^\perp)$.

\begin{pro} \label{pro:pointed}
Let $\pp \in \sign(S^\perp)$. % (and hence $C$ is pointed),
If $\Fc$ is surjective, % for all $c>0$. 
then $\pp \in \sign(\tilde S^\perp)$. % (and hence $\tilde C$ is pointed).
\end{pro}
\begin{proof}
By surjectivity, the image of $\Fc$ contains points arbitrarily close to zero.
Hence, there is a  sequence $\Seq_k$ in $\R^d$ 
such that $\Fc(\Seq_k ) \to 0$ as $k \to \infty$.
Since $\pp \in \sign(S^\perp)= \sign(\im W^\trans)$, there is $y \in \R^d$ such that $y\mal w^i > 0$ for all $i \in [n]$.
Now,
\[
y\cdot\Fc(\Seq_k) = \sum_{i=1}^n c_i \, (y \cdot w^i) \e^{\tilde w^i \mal \Seq_k} 
\]
is a sum of positive terms converging to zero, and hence each term goes to zero.
This implies $\tilde w^i \mal \Seq_k < 0$ for large $k$, for all $i \in [n]$.
Hence,
\[
\pp = \sign(-\tilde W^\trans \Seq_k) \in \sign(\im \tilde W^\trans)_\zp = \sign(\tilde S^\perp)_\zp .
\]
\end{proof}

%%%%%%%%% %%%%%%%%% %%%%%%%%% %%%%%%%%% %%%%%%%%%

\subsection{Sign-vector conditions} \label{sec:sv}

In general, the simultaneous bijectivity of $\Fc$ for all $c>0$, 
in particular, condition (iii) in Theorem~\ref{bij},
cannot be characterized in terms of sign vectors of $S$ and $\tilde S$.

\begin{exa} \label{exa:sv}
Let $\Fc$ be given by the matrices
\[
\tilde W =
\begin{pmatrix}
1 & 1 & 0 & 0 &-1 & \pa \\
1 &-1 & 0 & 0 & 0 & 0 \\
0 & 0 & 1 &-1 & 0 & 0
\end{pmatrix}
\quad \text{and} \quad
W =
\begin{pmatrix}
0 & 0 & 1 & 1 &-1 & 0 \\
1 &-1 & 0 & 0 & 0 &-1 \\
0 & 0 & 1 &-1 & 0 & 0
\end{pmatrix} ,
\]
involving the parameter $\pa > 0$.
Obviously, $\tilde C = C = \R^3$.
For $\pa=1$ or $\pa\in[2,\infty)$, the map $\Fc$ is injective for all $c>0$, but not bijective,
whereas for $\pa\in(0,1)$ or $\pa\in(1,2)$, the map $\Fc$ is bijective for all $c>0$.
Clearly, the sign vectors $\sign(\tilde S) = \sign(\ker \tilde W)$ do not depend on $\pa$
and hence cannot characterize bijectivity.
\end{exa}

In general, condition (iii) depends on the subspace $\tilde S$ itself.
Still, 
\begin{itemize}
\item 
condition (iii) holds trivially % (since $\sign(S)_\zp = \{0\}$).
if $\pp \in \sign(S^\perp)$, % (and hence $C$ is pointed), 
see Section~\ref{sec:special},
% and bijectivity is characterized by conditions (i) and (ii), that is, in terms of sign vectors (of $S$ and $\tilde S$).
\item
there is a (weakest) condition (\iiis) in terms of sign vectors of $S$ and $\tilde S$ {\em sufficient} for nondegeneracy,
see Proposition~\ref{iii}, and
\item
there is a {\em sufficient} condition for nondegeneracy using faces of the Newton polytope $\tilde P$,
% characterized by sign vectors (however, not of $S$ and $\tilde S$),
see Proposition~\ref{newton}.
(Thereby, faces of $\tilde P$ correspond to nonnegative sign vectors of an affine subspace related to $\tilde S$.)
\end{itemize}

\begin{pro} \label{iii}
Let $S,\tilde S$ be subspaces of $\R^n$.
If
\begin{itemize}
\item[(\iiis)]
for all $\tilde \tau \in \sign(\tilde S^\perp)$ with $\tilde \tau^+ \neq \emptyset$,
\begin{itemize}
\item[$\bullet$]
there is no $\pi \in \sign(S)_\zp$ with $\pi^+= \tilde \tau^+$
\item[$\bullet$]
or there is no $\rho \in \sign(S)$ with $\tilde \tau^+ \cup \tilde \tau^- \subseteq \rho^+$
% (That is, there is a nonzero $\tau \in \sign(S^\perp)_\zp$ with $\tilde \tau^0 \subseteq \tau^0$.)
\end{itemize}
\end{itemize}
then the pair $(S,\tilde S)$ is \nondeg. That is, (\iiis) $\Rightarrow$ (iii).
\end{pro}
\begin{proof}
Assume that $(S,\tilde S)$ is \degdeg,
in particular,
that $z \in \tilde S^\perp$ with a positive component violates nondegeneracy, 
and let $\tilde \tau = \sign(z) \in \sign(\tilde S^\perp)$, where $\tilde \tau^+ \neq \emptyset$.
\par
For every index set $I = \{ i \mid z_i = \la \}$ with $\la>0$,
the sign vector $\pi \in \{0,+\}^n$ with $\pi^+ = I$ satisfies $\pi \in \sign(S)_\zp$.
Clearly, the index sets $I$ cover $\tilde \tau^+ = \{i \mid z_i > 0\}$
and, by composition, there is $\pi \in \sign(S)_\zp$ with $\pi^+ = \tilde \tau^+$.
\par
Further, there is no nonzero $\tau \in \sign(S^\perp)_\zp$ such that $\tilde \tau^0 \subseteq \tau^0$,
that is, $\tau \le \abs{\tilde \tau}$.
Thereby, $\abs{\tilde \tau} \in \{0,+\}^n$ 
with $\abs{\tilde \tau}^0 = \tilde \tau^0$ and $\abs{\tilde \tau}^+ = \tilde \tau^+ \cup \tilde \tau^-$.
By Corollary~\ref{minty1} in Appendix~\ref{ta},
there is $\rho \in \sign(S)$ such that $\rho \ge \abs{\tilde \tau}$,
that is, $\abs{\tilde \tau}^+ \subseteq \rho^+$.
\end{proof}

Finally, we formulate a sufficient condition for nondegeneracy using faces of the Newton polytope $\tilde P = \conv \tilde W$, the convex hull of the columns of $\tilde W$.
A face $\tilde f$ of $\tilde P$ with $\tilde I = \{ i \mid \tilde w^i \in \tilde f \}$ 
corresponds to a supporting affine hyperplane with normal vector $x \in \R^d$ and $\la \in \R$
such that $\tilde w^i \mal x = \la$ for $i \in \tilde I$ and $\tilde w^i \mal x < \la$ otherwise;
that is, $\tilde I = \ixl{x}{\la}$.
It further corresponds to $z = \tilde W^\trans x \in \tilde S^\perp$, where $\tilde I = \{ i \mid z_i = \la \}$.
If $\la>0$, we call the face $\tilde f$ of $\tilde P$ {\em positive}, and $z \in \tilde S^\perp$ has a positive component.

\begin{pro} \label{newton}
Let $S,\tilde S$ be subspaces of $\R^n$,
$\tilde W \in \R^{d \times n}$ be a matrix with full rank such that $\ker \tilde W = \tilde S$,
and $\tilde P = \conv \tilde W$ be the Newton polytope.
The pair $(S,\tilde S)$ is \nondeg, if,
for every positive face $\tilde f$ of $\tilde P$ with $\tilde I = \{ i \mid \tilde w^i \in \tilde f \}$,
the sign vector $\pi \in \{ 0,+ \}^n$ with $\pi^+ = \tilde I$
satisfies $\pi \not\in \sign(S)_\zp$.
\end{pro}
\begin{proof}
Let $z \in \tilde S^\perp$ have a positive component,
$\la = \max_i z_i > 0$, and $\tilde I = \{ i \mid z_i = \la \}$.
Then $z$ corresponds to a positive face $\tilde f$ of $\tilde P$ with $\tilde I = \{ i \mid \tilde w^i \in \tilde f \}$.
If the sign vector $\pi \in \{ 0,+ \}^n$ with $\pi^+ = \tilde I$
satisfies $\pi \not\in \sign(S)_\zp$,
then $z$ is \nondeg, by definition.
\end{proof}

%%%%%%%%% %%%%%%%%% %%%%%%%%% %%%%%%%%% %%%%%%%%%
%%%%%%%%% %%%%%%%%% %%%%%%%%% %%%%%%%%% %%%%%%%%%
%%%%%%%%% %%%%%%%%% %%%%%%%%% %%%%%%%%% %%%%%%%%%

\section{Robustness of bijectivity} \label{sec:robust}

We study the robustness of the simultaneous bijectivity of $\Fc$ for all $c>0$
with respect to small perturbations of the exponents $\tilde W$ or/and the coefficients~$W$,
corresponding to small perturbations of the subspaces $\tilde S$ and $S$ (in the Grassmannian).

The set of all $n-d$ dimensional subspaces $S$ of $\R^n$ is the Grassmann manifold of rank $n-d$.
It is a  compact, connected smooth manifold of dimension $d(n-d)$, 
see e.g. \cite[Chapter IV.7]{GlazmanLjubic1974}.
There are many metrics on the Grassmannian that generate the same topology, 
for example, two subspaces $S$ and $\tilde S$ are close 
if and only if,
for all $x \in S$ with $\abs{x}=1$, there exists $\tilde x \in \tilde S$ close to $x$, and the other way round.

\subsection{Perturbations of the exponents} \label{sec:exp}

First, we consider small perturbations of the subspace $\tilde S$, corresponding to the exponents $\tilde W$ in $\Fc$.
As it turns out, the closure of $\sign(\tilde S)$ plays an important role.

\begin{dfn}
Let $T \subseteq \{-,0,+\}^n$. We define its {\em closure}
\[
\overline{T} = \{ \tau \in \{-,0,+\}^n \mid \tau \le \rho \text{ for some } \rho \in T \} .
\]
\end{dfn}
Clearly,  
$T_1 \subseteq \overline{T_2}$ implies $\overline{T_1} \subseteq \overline{T_2}$.

\begin{lem} \label{pert}
Let $S$ be a subspace of $\R^n$ and $S_\eps$ be a small perturbation. % (in the Grassmannian).
Then $\sign(S) \subseteq \overline{\sign(S_\eps)}$.
\end{lem}
\begin{proof}
Let $\pi \in \sign(S)$ and a corresponding $x \in S$ with $\pi = \sign(x)$.
Then there is $x_\eps \in S_\eps$ close to $x$.
For a small enough perturbation $S_\eps$, nonzero components keep their signs (but zero components can become nonzero),
that is, $\sign(x) \le \sign(x_\eps)$. Hence, $\pi \in \overline{\sign(S_\eps)}$.
\end{proof}

We start by studying injectivity.

\begin{lem} \label{pert_inj_cc}
Let $S,\tilde S$ be subspaces of $\R^n$. If $\sign(S) \cap \sign(\tilde S_\eps^\perp) = \{0\}$ for all small perturbations $\tilde S_\eps$, % (in the Grassmannian),
then $\sign(S) \subseteq \overline{\sign(\tilde S)}$.
\end{lem}
\begin{proof}
Suppose $\sign(S) \subseteq \overline{\sign(\tilde S)}$ does not hold.
Then there is a nonzero sign vector $\pi \in \sign(S)$ with $\pi \notin \overline{\sign(\tilde S)}$.
We will find a small perturbation $\tilde S_\eps$ such that $\pi \in \sign(\tilde S^\perp_\eps)$ and hence $\sign(S) \cap \sign(\tilde S_\eps^\perp) = \{0\}$ is violated.
\par
By Corollary~\ref{minty1} in Appendix~\ref{ta},
the nonexistence of $\rho \in \sign(\tilde S)$ with $\rho \ge \pi$
implies the existence of a nonzero $\tilde \tau \in \sign(\tilde S^\perp)$ with $\tilde \tau \le \pi$.
If $\tilde \tau = \pi$, then $\pi \in \sign(\tilde S^\perp)$, as desired.
Otherwise, let $\tilde \tau = \sign(x)$ for $x \in \tilde S^\perp$.
We find a perturbation $x_\eps = x+\eps e$ with $\eps>0$ small and $e \in \R^n$ such that $\sign(x_\eps) = \pi$.
In particular, we choose $e_i=1$ if $x_i=0$ and $i \in \pi^+$, $e_i=-1$ if $x_i=0$ and $i \in \pi^-$, and $e_i=0$ otherwise.
Then, we rescale $x_\eps$ such that $|x_\eps|=|x|$.
Finally, we find an orthogonal matrix $U \in \R^{n \times n}$ (close to the identity) such that $U x = x_\eps$.
Then $x_\eps = Ux \perp U\tilde S = \tilde S_\eps$, that is, $x_\eps \in \tilde S_\eps^\perp$ and $\pi \in \sign(\tilde S_\eps^\perp)$, as desired.
\end{proof}

\begin{lem} \label{cc_i}
Let $S,\tilde S$ be subspaces of $\R^n$.
If $\sign(S) \subseteq \overline{\sign(\tilde S)}$,
then $\sign(S) \cap \sign(\tilde S^\perp) = \{0\}$.
\end{lem}
\begin{proof}
Assume there exists a nonzero $\tilde \tau \in \sign(S) \cap \sign(\tilde S^\perp)$.
If $\sign(S) \subseteq \overline{\sign(\tilde S)}$, then there exists $\rho \in \sign(\tilde S)$ with $\rho \ge \tilde \tau$.
In particular, $\tilde \tau \cdot \rho \neq 0$, thereby contradicting $\tilde \tau \in \sign(\tilde S)^\perp = \sign(\tilde S)^\perp$ and $\rho \in \sign(\tilde S)$. Cf.~Corollary~\ref{orthogonal} in Appendix~\ref{ta}.
\end{proof}

\begin{pro} \label{pert_inj=cc}
Let $S,\tilde S$ be subspaces of $\R^n$.
Then $\sign(S) \cap \sign(\tilde S_\eps^\perp) = \{0\}$ for all small perturbations $\tilde S_\eps$ % (in the Grassmannian)
if and only if $\sign(S) \subseteq \overline{\sign(\tilde S)}$.
\end{pro}
\begin{proof}
($\Rightarrow$): By Lemma~\ref{pert_inj_cc}.
\par
($\Leftarrow$): Assume $\sign(S) \subseteq \overline{\sign(\tilde S)}$.
By Lemma~\ref{pert}, $\sign(\tilde S) \subseteq \overline{\sign(\tilde S_\eps)}$ for all small perturbations $\tilde S_\eps$
which implies $\overline{\sign(\tilde S)} \subseteq \overline{\sign(\tilde S_\eps)}$.
Hence, $\sign(S) \subseteq \overline{\sign(\tilde S_\eps)}$.
By Lemma~\ref{cc_i}, $\sign(S) \cap \sign(\tilde S_\eps^\perp) = \{0\}$.
\end{proof}

\begin{cor} \label{cc_cc}
Let $S,\tilde S$ be subspaces of $\R^n$.
Then
\[
\sign(S) \subseteq \overline{\sign(\tilde S)} \quad \text{if and only if} \quad \sign(S^\perp) \subseteq \overline{\sign(\tilde S^\perp)} .
\]
\end{cor}
\begin{proof}
By Corollary~\ref{inj_cor}, $\sign(S) \cap \sign(\tilde S_\eps^\perp) = \{0\}$ is equivalent to $\sign(S^\perp) \cap \sign(\tilde S_\eps) = \{0\}$.
By Proposition~\ref{pert_inj=cc} twice, the former statement (for all small perturbations $\tilde S_\eps$)
is equivalent to $\sign(S) \subseteq \overline{\sign(\tilde S)}$
and the latter to $\sign(S^\perp) \subseteq \overline{\sign(\tilde S^\perp)}$.
\end{proof}

In terms of the map $\Fc$ (and the associated subspaces $S$ and $\tilde S$), Proposition~\ref{pert_inj=cc} states that
\[
\parbox{0.38\textwidth}{$\Fc$ is injective for all $c>0$ \\ and all small perturbations $\tilde S_\eps$}
\quad \Leftrightarrow \quad
\sign(S) \subseteq \overline{\sign(\tilde S)} .
\]

In Proposition~\ref{cc_bij} and Theorem~\ref{pert_tildeS} below,
we will show that
\[
\sign(S) \subseteq \overline{\sign(\tilde S)}
\quad \Rightarrow \quad
\parbox{0.34\textwidth}{$\Fc$ is bijective for all $c>0$}
\]
and
\[
\parbox{0.38\textwidth}{$\Fc$ is bijective for all $c>0$ \\ and all small perturbations $\tilde S_\eps$}
\quad \Leftrightarrow \quad
\sign(S) \subseteq \overline{\sign(\tilde S)} .
\]

First, we prove that the closure condition
\begin{equation} \tag{cc} \label{cc}
\sign(S) \subseteq \overline{\sign(\tilde S)} 
\end{equation}
implies the bijectivity of $\Fc$ for all $c>0$,
that is, conditions~(i), (ii), and (iii) in Theorem~\ref{bij}.
For an alternative proof, using differential topology, see~\cite{CMPY2018}.

\begin{pro} \label{cc_bij}
If $\sign(S) \subseteq \overline{\sign(\tilde S)}$,
then the map $\Fc$ is a diffeomorphism for all $c>0$.
\end{pro}
\begin{proof}
\par
\eqref{cc} $\Rightarrow$ (i): By Lemma~\ref{cc_i}.
\par
\eqref{cc} $\Rightarrow$ (ii):
\par
Assume $\neg$(ii), that is, the existence of a nonzero $\tilde \tau \in \sign(\tilde S^\perp)_\zp$ with $\tau \not\le \tilde \tau$ for all nonzero $\tau \in \sign(S^\perp)_\zp$,
in fact, for all nonzero $\tau \in \sign(S^\perp)$.
By Corollary~\ref{minty1} in Appendix~\ref{ta}, the nonexistence of a nonzero $\tau \in \sign(S^\perp)$ 
with $\tau \le \tilde \tau$
implies the existence of $\pi \in \sign(S)$ with $\pi \ge \tilde \tau$.
\par
Now, if $\sign(S) \subseteq \overline{\sign(\tilde S)}$, then there exists $\rho \in \sign(\tilde S)$ with 
$\rho \ge \pi$ and hence $\rho \ge \tilde \tau$.
In particular, $\tilde \tau \cdot \rho \neq 0$, thereby contradicting $\tilde \tau \in \sign(\tilde S^\perp) = \sign(\tilde S)^\perp$ and $\rho \in \sign(\tilde S)$.
\par
\eqref{cc} $\Rightarrow$ (\iiis) in Proposition~\ref{iii}:
\par
Assume $\neg$(\iiis),
that is,
the existence of $\tilde \tau \in \sign(\tilde S^\perp)$ with $\tilde \tau^+ \neq \emptyset$, $\pi \in \sign(S)_\zp$ with $\pi^+ = \tilde \tau^+$,
and $\rho \in \sign(S)$ with $\tilde \tau^+ \cup \tilde \tau^- \subseteq \rho^+$.
By composition, 
$\pi' = \pi \circ (-\rho) \in \sign(S)$,
where $\pi'_i = +$ for $i \in \tilde \tau^+$ and $\pi'_i = -$ for $i \in \tilde \tau^-$,
that is, 
$\pi' \ge \tilde \tau$.
\par
Now, if $\sign(S) \subseteq \overline{\sign(\tilde S)}$, then there exists $\rho' \in \sign(\tilde S)$ with 
$\rho' \ge \pi'$ and hence $\rho' \ge \tilde \tau$.
In particular, $\tilde \tau \cdot \rho' \neq 0$, thereby contradicting $\tilde \tau \in \sign(\tilde S^\perp) = \sign(\tilde S)^\perp$ and $\rho' \in \sign(\tilde S)$.
\end{proof}

However, the closure condition~\eqref{cc} is not necessary for bijectivity.
Recall that there is a (weakest) sign-vector condition sufficient for bijectivity, 
involving conditions (i), (ii), and (\iiis) in Proposition~\ref{iii}.

\begin{exa} \label{exa:cc}
Let $\Fc$ be given by the matrices
\[
\tilde W =
\begin{pmatrix}
1 & 0 &-1
\end{pmatrix}
\quad \text{and} \quad
W =
\begin{pmatrix}
1 & 1 &-1 
\end{pmatrix}  .
\]
Obviously, $\tilde C = C = \R$.
Now, for $\tau = (+,+,-)^\trans \in \sign(\im W^\trans) = \sign(S^\perp)$,
there is no $\tilde \tau \in \sign(\im \tilde W^\trans) = \sign(\tilde S^\perp)$ with $\tilde \tau \ge \tau$.
Hence, $\sign(S^\perp) \not \subseteq \overline{\sign(\tilde S^\perp)}$, that is, the closure condition~\eqref{cc} does not hold.
Still, there is no nonzero $\pi \in \sign(\ker W)_\zp = \sign(S)_\zp$,
and hence condition (\iiis) holds.
Further, conditions (i) and (ii) hold, and $\Fc$ is bijective for all $c>0$.
\end{exa}

In fact, the closure condition~\eqref{cc} is equivalent to bijectivity for all small perturbations $\tilde S_\eps$.

\begin{thm} \label{pert_tildeS}
The map $\Fc$ is a diffeomorphism for all $c>0$ and all small perturbations $\tilde S_\eps$ % (in the Grassmannian)
if and only if $\sign(S) \subseteq \overline{\sign(\tilde S)}$.
\end{thm}
\begin{proof}
By Lemma~\ref{pert}, $\sign(S) \subseteq \overline{\sign(\tilde S)}$ implies $\sign(S) \subseteq \overline{\sign(\tilde S_\eps)}$ for all small perturbations~$\tilde S_\eps$.
By Proposition~\ref{cc_bij}, the latter implies the bijectivity of $\Fc$ for all $c>0$ and all small perturbations~$\tilde S_\eps$.

Bijectivity implies injectivity, that is, $\sign(S) \cap \sign(\tilde S_\eps^\perp) = \{0\}$, for all small perturbations~$\tilde S_\eps$.
By Lemma~\ref{pert_inj_cc}, the latter implies $\sign(S) \subseteq \overline{\sign(\tilde S)}$.
\end{proof}

Corollary~\ref{inj_det} relates {\em chirotopes} (signs of maximal minors of $W$ and $\tilde W$)
to {\em vectors} (sign vectors of $S = \ker W$ and $\tilde S = \ker \tilde W$).
By varying over all small perturbations $\tilde S_\eps$, % (in the Grassmannian),
we obtain the following result.

\begin{pro} \label{cc_det}
Let $S,\tilde S$ be subspaces of $\R^n$ of dimension $n-d$ (with $d \le n$).
For every $W, \tilde W \in \R^{d \times n}$ (with full rank $d$) such that $S = \ker W$ and $\tilde S = \ker \tilde W$,
the following statements are equivalent.
\begin{enumerate}
\item
$\sign(S) \subseteq \overline{\sign(\tilde S)}$.
\item
$\det(W_I) \neq 0$ implies $\det(W_I) \det(\tilde W_I) > 0$ for all subsets $I \subseteq [n]$ of cardinality~$d$
(or `$< 0$' for all $I$).
\end{enumerate}
\end{pro}
\begin{proof}
By Proposition~\ref{pert_inj=cc}, statement~1 is equivalent to 
$\sign(S) \cap \sign(\tilde S_\eps^\perp) = \{0\}$ for all small perturbations $\tilde S_\eps$. % (in the Grassmannian).
By Corollary~\ref{inj_det}, this is equivalent to
\begin{itemize}
\item[]
$\det(W_I) \det(\tilde W_{\eps,I}) \ge 0$ for all $I \subseteq [n]$ of cardinality $d$ (or `$\le 0$' for all $I$)
and $\det(W_I) \det(\tilde W_{\eps,I}) \neq 0$ for some~$I$, \\
for all small perturbations $\tilde W_\eps$ of $\tilde W$.
\end{itemize}
This is equivalent to statement~2,
thereby using that $\det(\tilde W_I) = 0$ implies $\det(\tilde W_{\eps_1,I}) < 0$ and $\det(\tilde W_{\eps_2,I}) > 0$ 
for some small perturbations~$\tilde W_{\eps_1}$ and~$\tilde W_{\eps_2}$.
\end{proof}

Now we can extend Theorem~\ref{pert_tildeS}.
In particular, we can characterize the bijectivity of $\Fc$ for all $c > 0$ and all small perturbations $\tilde S_\eps$
not only in terms of sign vectors, but also in terms of maximal minors.

\begin{cor} 
% Let $\Fc$ be as defined above. % (and $d=\tilde d$).
The following statements are equivalent:
\begin{enumerate}
\item
$\Fc$ is a diffeomorphism for all $c>0$ and all small perturbations $\tilde S_\eps$. % (in the Grassmannian).
\item
$\sign(S) \subseteq \overline{\sign(\tilde S)}$.
\item
$\det(W_I) \neq 0$ implies $\det(W_I) \det(\tilde W_I) > 0$ for all subsets $I \subseteq [n]$ of cardinality~$d$
(or `$< 0$' for all $I$).
\end{enumerate}
\end{cor}
\begin{proof}
(1 $\Leftrightarrow$ 2): 
By Theorem~\ref{pert_tildeS}. 
(2 $\Leftrightarrow$ 3): 
By Proposition~\ref{cc_det}.
\end{proof}

%%%%%%%%% %%%%%%%%% %%%%%%%%% %%%%%%%%% %%%%%%%%%

\subsection{Perturbations of the coefficients} \label{sec:coeff}

Next, we consider small perturbations of the subspace $S$, corresponding to the coefficients $W$ in $\Fc$.
We start by studying injectivity.
By Corollary~\ref{inj_cor}, the perturbed injectivity condition $\sign(S_\eps) \cap \sign(\tilde S^\perp) = \{0\}$ is equivalent to $\sign(\tilde S) \cap \sign(S_\eps^\perp) = \{0\}$.
By exchanging the roles of $S$ and $\tilde S$ in Proposition~\ref{pert_inj=cc},
we immediately obtain the desired result.

\begin{cor} \label{pert_inj=cc'}
Let $S,\tilde S$ be subspaces of $\R^n$.
Then $\sign(S_\eps) \cap \sign(\tilde S^\perp) = \{0\}$ for all small perturbations $S_\eps$ % (in the Grassmannian)
if and only if $\sign(\tilde S) \subseteq \overline{\sign(S)}$.
\end{cor}

The closure condition
\begin{equation} \tag{cc'} \label{cc'}
\sign(\tilde S) \subseteq \overline{\sign(S)}
\end{equation}
is equivalent to $\sign(\tilde S^\perp) \subseteq \overline{\sign(S^\perp)}$, by Corollary~\ref{cc_cc}.
As opposed to (cc),
it does not imply bijectivity, in fact, it implies conditions (i) and (iii) in Theorem~\ref{bij}, but not condition (ii).

\begin{pro} \label{cc'_i_iii}
If $\sign(\tilde S) \subseteq \overline{\sign(S)}$,
then conditions~(i) and (iii) in Theorem~\ref{bij} hold.
\end{pro}
\begin{proof}
\par
\eqref{cc'} $\Rightarrow$ (i): By Corollary~\ref{pert_inj=cc'}.
\par
\eqref{cc'} $\Rightarrow$ (\iiis) in Proposition~\ref{iii}:
\par
Assume $\neg$(\iiis) and hence 
the existence of $\tilde \tau \in \sign(\tilde S^\perp)$ and $\pi \in \sign(S)_\zp$ with $\tilde \tau^+ = \pi^+ \neq \emptyset$,
in particular, 
$\tilde \tau \ge \pi$.
Now, if $\sign(\tilde S) \subseteq \overline{\sign(S)}$, then there exists $\rho \in \sign(S^\perp)$ with 
$\rho \ge \tilde \tau$ and hence $\rho \ge \pi$. In particular,
$\pi \cdot \rho \neq 0$, thereby contradicting $\pi \in \sign(S)$ and $\rho \in \sign(S^\perp) = \sign(S)^\perp$.
\end{proof}

\begin{exa}
Let $\Fc$ be given by the matrices
\[
\tilde W =
\begin{pmatrix}
1 & 0 &-1 \\
0 & 1 & 0
\end{pmatrix}
\quad \text{and} \quad
W =
\begin{pmatrix}
1 & 1 & 0 \\
0 & 1 & 1
\end{pmatrix}  .
\]
Obviously, $\tilde C = \R \times \R_{\ge0}$ and $C = \R_{\ge0}^2$.
Now, $\tilde S = \ker \tilde W = \im (1,0,1)^\trans$, $S = \ker W = \im (1,-1,1)^\trans$, and hence $\sign(\tilde S) \subseteq \overline{\sign(S)}$.
However, $\sign(\tilde S^\perp)_\zp = \{ (0,0,0)^\trans$, $(0,+,0)^\trans \}$, $\sign(S^\perp)_\zp = \{ (0,0,0)^\trans$, $(0,+,+)^\trans$, $(+,+,0)^\trans \}$, 
and hence condition (ii) does not hold.
\end{exa}

Interestingly, conditions (cc') and (ii) imply the equality of the face lattices of $C$ and $\tilde C$.

\begin{pro} \label{pro:fl}
If $\sign(\tilde S) \subseteq \overline{\sign(S)}$ and condition (ii) in Theorem~\ref{bij} holds,
then $\sign(S^\perp)_\zp = \sign(\tilde S^\perp)_\zp$.
\end{pro}
\begin{proof}
Recall that, by the proof of Proposition~\ref{cc_bij}, (cc) implies (ii); analogously, (cc') implies 
\begin{itemize}
\item[(ii')]
for every nonzero $\tau \in \sign(S^\perp)_\zp$, 
there is a nonzero $\tilde \tau \in \sign(\tilde S^\perp)_\zp$ such that $\tilde \tau \le \tau$.
\end{itemize}

On the one hand,
let $\tau \in \sign(S^\perp)_\zp$ have minimal support.
By (ii'), there is a nonzero $\tilde \tau \in \sign(\tilde S^\perp)_\zp$ such that $\tilde \tau \le \tau$.
By (ii), there is a nonzero $\tau' \in \sign(S^\perp)_\zp$
such that $\tau' \le \tilde \tau$.
Altogether, 
$\tau' \le \tilde \tau \le \tau$.
Now, $\tau' = \tau$,
since $\tau$ has minimal support,
and hence $\tilde \tau = \tau$.
That is, 
there is a {\em unique} nonzero $\tilde \tau \in \sign(\tilde S^\perp)_\zp$ (namely $\tilde \tau = \tau$) such that $\tilde \tau \le \tau$.
In particular, $\tilde \tau$ has minimal support.

On the other hand,
let $\tilde \tau \in \sign(\tilde S^\perp)_\zp$ have minimal support.
By an analogous argument,
there is a {\em unique} nonzero $\tau \in \sign(S^\perp)_\zp$ (namely $\tau = \tilde \tau$) such that $\tau \le \tilde \tau$.
In particular, $\tilde \tau$ has minimal support.
Hence, elements of $\sign(S^\perp)_\zp$ and $\sign(\tilde S^\perp)_\zp$ with minimal support are in one-to-one correspondence.
Finally,
every nonzero, nonnegative sign vector of a subspace is the composition of nonnegative sign vectors with minimal support, cf.~Theorem~\ref{conformal} in Appendix~\ref{sv}.
Hence, $\sign(S^\perp)_\zp = \sign(\tilde S^\perp)_\zp$.
\end{proof}

It remains to study the robustness of condition (ii).

\begin{lem} \label{ii}
If, for all small perturbations $S_\eps$, % (in the Grassmannian),
the map $\Fc$ is surjective and condition (ii) in Theorem~\ref{bij} holds, 
then either $C = \tilde C = \R^d$ or $\pp \in \sign(S^\perp) \cap \sign(\tilde S^\perp)$. % (and hence $C$ is pointed).
\end{lem}
\begin{proof}
If neither $C=\R^d$ nor $\pp \in \sign(S^\perp)$, then $C$ has a nontrivial lineality space.
On the one hand, 
there is a small perturbation $S_{\eps_1}$ such that%
\footnote{
Let $L \subset [n]$ be the indices of the vectors $w^i$ in the lineality space and $I = [n] \setminus L$.
Hence, there are $c_i>0$ for $i \in L$ such that $\sum_{i \in L} c_i w^i = 0$ and $\sum_{i \in L} c_i = 1$.
Consider small perturbations $S_\eps$ as follows: 
$w^i_\eps = w^i$ for $i \in I$ and $w^i_\eps = w^i - \eps \sum_{j \in I} w^j$ for $i \in L$, where $\eps>0$.
Then, $\sum_{i \in L} c_i w_\eps^i + \sum_{i \in I} \eps \, w_\eps^i = 0$,
and hence $\pp \in \sign(\ker W_\eps) = \sign(S_\eps)$, that is, $C_\eps=\R^d$.}
%Let $L \subset [n]$ be the indices of the vectors $w^i$ in the lineality space and $I = [n] \setminus L$.
%Consider small perturbations $S_\eps$ as follows: 
%$w^i_\eps = w^i$ for $i \in I$ and $w^i_\eps = w^i - \eps \sum_{i \in I} w^i$ for $i \in L$, where $\eps>0$.
%Clearly, $\tau = \sign(W^\trans x) \in \sign(S^\perp)$ and $\tau_\eps = \sign(W_\eps^\trans x) \in \sign(S_\eps^\perp)$
%imply $\tau_\eps \ge \tau$,
%in particular, $\tau \notin \sign(S^\perp)_\zp$ implies $\tau_\eps \notin \sign(S_\eps^\perp)_\zp$.
%Now, consider a nonzero $\tau \in \sign(S^\perp)_\zp$, 
%that is, $\tau_i\ge0$ for $i \in I$ (with $\tau_i=+$ for some $i \in I$) and $\tau_i=0$ for $i \in L$.
%Then,  $(\tau_\eps)_i = \tau_i \ge 0$ for $i \in I$
%and $(\tau_\eps)_i = -$ for $i \in L$, 
%since $w_\eps^i \cdot x = w^i \cdot x - \bar\eps \sum_{i \in I} w^i \cdot x < 0$,
%that is, $\tau_\eps \notin \sign(S_\eps^\perp)_\zp$.
%Hence $\sign(S_\eps^\perp)_\zp = \{0\}$, that is, $C_\eps=\R^d$.}
$C_{\eps_1} = \R^d$;
hence $\tilde C = \R^d$, by (ii).
On the other hand, 
there is a small perturbation $S_{\eps_2}$ such that $\pp \in \sign(S_{\eps_2}^\perp)$;
hence $\pp \in \sign(\tilde S^\perp)$, by Proposition~\ref{pro:pointed}.
A contradiction.
\par
If $C = \R^d$, then $\tilde C = \R^d$, by~(ii).
If $\pp \in \sign(S^\perp)$, then $\pp \in \sign(\tilde S^\perp)$, by Proposition~\ref{pro:pointed}.
\end{proof}

That is, condition (ii) is robust only in two extreme cases regarding the geometry of $C = \cone(W)$.
We consider the case $\pp \in \sign(S^\perp)$ % (and hence $C$ is pointed) 
separately.

We call $C$ {\em robustly generated} if either $d=1$
or, on every extreme ray of $C$, there lies a unique vector $w^i$,
and all other vectors lie in the interior.
In terms of sign vectors, $C$ is robustly generated if
\begin{itemize}
\item[]
%a nonzero $\tau \in \sign(S^\perp)_\zp \setminus \{\pp\}$ has maximal support if and only if $|\tau^0| = 1$.
%either $d=1$ or
a nonzero $\tau \in \sign(S^\perp)_\zp$ has minimal support if and only if,
for every $i \in \tau^0$, there exists $\hat\tau \in \sign(S^\perp)_\zp$ with $\hat\tau^0 = \{i\}$. % (and maximal support)
\end{itemize}
In this case,
$\sign(S_\eps^\perp)_\zp = \sign(S^\perp)_\zp$ for all small perturbations $S_\eps$,
and condition (ii) is robust.
In fact, (ii) being robust implies $C$ being robustly generated.

\begin{lem} \label{ii_pointed}
Let $\pp \in \sign(S^\perp)$ and $\sign(S^\perp)_\zp = \sign(\tilde S^\perp)_\zp$. % (and hence $C$ is pointed).
If condition (ii) in Theorem~\ref{bij} holds
for all small perturbations $S_\eps$, % (in the Grassmannian),
then
$C$ and $\tilde C$ are robustly generated.
\end{lem}
\begin{proof}
Let $d>1$.
Assume that $C$ is not robustly generated,
and let $f$ be a maximal proper face,
characterized by $\tau \in \sign(S^\perp)_\zp$ with minimal support,
such that $w^j \in f$ %(that is, $\tau_j=0$) 
for some $j \in [n]$, but $w^j$ is not needed to generate $f$.
Further let $\tilde f$ be the corresponding maximal proper face of $\tilde C$,
characterized by $\tilde \tau = \tau \in \sign(\tilde S^\perp)_\zp$ with minimal support.
In particular, $\tau_j = \tilde \tau_j = 0$.

Now, consider a small perturbation $S_\eps$
such that $w_\eps^j \in C^\circ$ and $w_\eps^i = w^i$ for $i \neq j$ (and hence $C_\eps = C$).
Then, $\tau'_j=+$ for all $\tau' \in \sign(S_\eps^\perp)_\zp$,
and there is no $\tau' \in \sign(S_\eps^\perp)_\zp$ with $\tau' \le \tilde \tau$, contradicting~(ii) for $\tilde \tau$.
\end{proof}

Finally, the closure condition~\eqref{cc'} together with sign-vector conditions regarding the geometry of the cones $C$ and $\tilde C$ is equivalent to bijectivity for all small perturbations~$S_\eps$.

\begin{thm} \label{pert_S}
The map $\Fc$ is a diffeomorphism for all $c>0$ and all small perturbations $S_\eps$ % (in the Grassmannian)
if and only if 
$\sign(\tilde S) \subseteq \overline{\sign(S)}$ 
and
\begin{itemize}
\item[]
either $C = \tilde C = \R^d$ \\
or $\pp \in \sign(S^\perp) \cap \sign(\tilde S^\perp)$, %(and hence $C$ and $\tilde C$ are pointed) 
$\sign(S^\perp)_\zp = \sign(\tilde S^\perp)_\zp$, 
and $C$ and $\tilde C$ are robustly generated.
\end{itemize}
\end{thm}
\begin{proof}
By Theorem~\ref{bij}, the simultaneous bijectivity of $\Fc$ for all $c>0$ is equivalent to conditions (i), (ii), and (iii) in Theorem~\ref{bij}.

By Corollary~\ref{pert_inj=cc'}, condition (i), that is, 
$\sign(S_\eps) \cap \sign(\tilde S^\perp) = \{0\}$, for all small perturbations $S_\eps$,
is equivalent to $\sign(\tilde S) \subseteq \overline{\sign(S)}$.

Now assume conditions (i), (ii), and (iii), for all small perturbations $S_\eps$.
By Proposition~\ref{pro:fl}, $\sign(S^\perp)_\zp = \sign(\tilde S^\perp)_\zp$.
By Lemma~\ref{ii},
either $C = \tilde C = \R^d$ or $\pp \in \sign(S^\perp) \cap \sign(\tilde S^\perp)$.
In the latter case, 
by Lemma~\ref{ii_pointed}, $C$ and $\tilde C$ are robustly generated.

Conversely, 
$\tilde C = \R^d$ (that is, $\sign(\tilde S^\perp)_\zp = \{0\}$)
trivially implies condition (ii) for all small perturbations $S_\eps$.
By Lemma~\ref{pert}, $\sign(\tilde S) \subseteq \overline{\sign(S)}$ implies $\sign(\tilde S) \subseteq \overline{\sign(S_\eps)}$ for all small perturbations~$S_\eps$,
and by Proposition~\ref{cc'_i_iii} (for $\tilde S$ and $S_\eps$), this implies condition (iii) for all small perturbations~$S_\eps$.

Finally, $\pp \in \sign(S^\perp)$,  $\sign(S^\perp)_\zp = \sign(\tilde S^\perp)_\zp$, and $C$ being robustly generated
imply $\pp \in \sign(S_\eps^\perp)$ and hence condition (iii), for all small perturbations $S_\eps$.
Further, they imply $\sign(S_\eps^\perp)_\zp = \sign(\tilde S^\perp)_\zp$
and hence condition (ii), for all small perturbations~$S_\eps$.
\end{proof}

%%%%%%%%% %%%%%%%%% %%%%%%%%% %%%%%%%%% %%%%%%%%%

\subsection{General perturbations} \label{sec:general}

Finally, we consider small perturbations of both subspaces, $S$ and $\tilde S$,
corresponding to the coefficients $W$ and the exponents $\tilde W$ in $\Fc$.

The next result relates {\em chirotopes} to {\em cocircuits} (sign vectors of $S^\perp = \im W^\trans$ and $\tilde S^\perp = \im \tilde W^\trans$ with minimal support).

\begin{lem} \label{chirotope_circuit}
Let $S,\tilde S$ be subspaces of $\R^n$ of dimension $n-d$ (with $d \le n$).
For every $W, \tilde W \in \R^{d \times n}$ (with full rank $d$) such that $S = \ker W$ and $\tilde S = \ker \tilde W$,
the following statements are equivalent.
\begin{enumerate}
\item
$\sign(S) = \sign(\tilde S)$, and a nonzero $\tau \in \sign(S^\perp)$ has minimal support if and only if $|\tau^0| = d-1$.
\item
$\det(W_I) \det(\tilde W_I) > 0$ for all subsets $I \subseteq [n]$ of cardinality~$d$
(or `$< 0$' for all $I$).
\end{enumerate}
\end{lem}
\begin{proof}
By using the standard chirotope/cocircuit translation for subspaces of~$\R^n$,
see Theorem~\ref{translation} in Appendix~\ref{sv}.
\end{proof}

As it turns out, 
all maximal minors of $W$ and $\tilde W$ being nonzero and having matching signs
is equivalent to bijectivity for all small perturbations $S_\eps$ and $\tilde S_{\tilde \eps}$.

\begin{thm} \label{pert_S_tildeS}
The following statements are equivalent:
\begin{enumerate}
\item
$\Fc$ is a diffeomorphism for all $c>0$ and all small perturbations $S_\eps$ and~$\tilde S_{\tilde \eps}$. % (in the Grassmannian).
\item
$\sign(S) = \sign(\tilde S)$, and a nonzero $\tau \in \sign(S^\perp)$ has minimal support if and only if $|\tau^0| = d-1$.
\item
$\det(W_I) \det(\tilde W_I) > 0$ for all subsets $I \subseteq [n]$ of cardinality~$d$
(or `$< 0$' for all $I$).
\end{enumerate}
\end{thm}
\begin{proof} 
(1 $\Rightarrow$ 3):
Statement~1 implies the injectivity of $\Fc$ for all $c>0$, that is, $\sign(S_\eps) \cap \sign(\tilde S_{\tilde \eps}^\perp) = \{0\}$,
for all small perturbations $S_\eps$, $\tilde S_{\tilde \eps}$.
By Corollary~\ref{inj_det}, this is equivalent to
\begin{itemize}
\item[]
$\det(W_{\eps,I}) \det(\tilde W_{\tilde \eps,I}) \ge 0$ for all $I \subseteq [n]$ of cardinality $d$ (or `$\le 0$' for all $I$)
and $\det(W_{\eps,I}) \det(\tilde W_{\tilde \eps,I}) \neq 0$ for some~$I$, \\
for all small perturbations $W_\eps$ of $W$ and $\tilde W_{\tilde \eps}$ of $\tilde W$. 
\end{itemize}
This is equivalent to statement~3.

(3 $\Rightarrow$ 1):
Statement~3 implies
\begin{itemize}
\item[]
$\det(W_{\eps,I}) \det(\tilde W_{\tilde \eps,I}) > 0$ for all $I \subseteq [n]$ of cardinality $d$ (or `$< 0$' for all $I$), \\
for all small perturbations $W_\eps$, $\tilde W_{\tilde \eps}$. 
\end{itemize}
By Lemma~\ref{chirotope_circuit}, this implies $\sign(S_\eps) = \sign(\tilde S_{\tilde \eps})$ and hence $\sign(S_\eps) \subseteq \overline{\sign(\tilde S_{\tilde \eps})}$,
for all small perturbations $W_\eps$, $\tilde W_{\tilde \eps}$. 
By Proposition~\ref{cc_bij}, this implies statement~1.

(2 $\Leftrightarrow$ 3): By Lemma~\ref{chirotope_circuit}.
\end{proof}

By Theorem~\ref{pert_S}, 
bijectivity for all $c>0$ and all small perturbations $S_\eps$ already implies that either $C=\tilde C=\R^d$ or $\pp \in \sign(S^\perp) \cap \sign(\tilde S^\perp)$.
In Theorem~\ref{pert_S_tildeS},
this follows from the second part of condition 2.
Assume that $C$ has a nontrivial lineality space of dimension $\ell$, generated by at least $\ell+1$ vectors~$w^i$.
Then, a maximal proper face, having dimension $d-1 = \ell+d'$, is generated by at least $(\ell+1)+d'=d$ vectors
and corresponds to a sign vector $\tau \in \sign(S^\perp)$ with minimal support, but $|\tau^0| \ge d$.

%%%%%%%%% %%%%%%%%% %%%%%%%%% %%%%%%%%% %%%%%%%%%

\section{Applications to Chemical Reaction Networks} \label{sec:appl}

As mentioned in the introduction, 
our work is motivated by the study of chemical reaction networks with generalized mass-action kinetics.
We present a derivation of our main problem (the characterization of bijectivity of families of exponential maps)
and applications of our main results,
in particular, Theorems~\ref{bij} and \ref{pert_tildeS}.

We start with an introduction to chemical reaction networks (with mass-action kinetics).
Thereby, we follow the graph-based approach introduced in \cite{mueller:regensburger:2014}; 
see also \cite{CMPY2018,JohnstonMuellerPantea2018}.

Consider the chemical reaction 
$
1 \ce{A} + 1 \ce{B}  \to \ce{C}
$
(with {\em stoichiometric coefficients} equal to 1).
Under the assumption of mass-action kinetics (MAK),
the reaction rate is given by 
$
v = k \, x_\ce{A}^1 x_\ce{B}^1
$
(with {\em kinetic orders} equal to 1),
where $k>0$ is the {\em rate constant}
and $x_\ce{A}, x_\ce{B} \ge 0$ are the concentrations of the chemical species $\ce{A},\ce{B}$.
Most importantly, the stoichiometric coefficients determine the kinetic orders.
Given $n$ species,
a general reaction
is written as
$
y \to y' ,
$
where $y,y' \in \R^n_{\ge0}$ are called (educt and product) {\em complexes},
and its rate is given by
$
v = k \, x^y ,
$
where $x^y = \prod_{i=1}^n {x_i}^{y_i}$ is a monomial in the species concentrations $x \in \R^n_{\ge0}$.
In a network,
an individual reaction $y \to y'$ contributes to the ODE for the species concentrations as
$
\dd{x}{t} = k \, x^y (y' - y) + \ldots \, .
$
Let $x=(x_\ce{A},x_\ce{B},x_\ce{C},x_\ce{D},\ldots)^\trans$.
For the reaction $\ce{A} + \ce{B}  \to \ce{C}$ above, one has
$y=(1,1,0,0,\ldots)^\trans$, $y'=(0,0,1,0,\ldots)^\trans$
and hence
$x^y=x_\ce{A}x_\ce{B}$, $y'-y = (-1,-1,1,0,\ldots)^\trans$. % and
%\[
%\dd{}{t} \begin{pmatrix} x_\ce{A} \\ x_\ce{B} \\ x_\ce{C} \\ x_\ce{D} \\ \vdots \end{pmatrix}
%= k \, x_\ce{A}x_\ce{B}
%\begin{pmatrix} -1 \\ -1 \\ 1 \\ 0 \\ \vdots \end{pmatrix}
%+ \ldots \, .
%\]

A {\em chemical reaction network} (CRN) is based on a directed graph $G=(V,E)$.
Every vertex $i \in V=\{1,\ldots,m\}$ is labeled with a complex $y(i) \in \R^n_{\ge0}$, 
and every edge $i \to i' \in E$ (representing a reaction) is labeled with a rate constant $k_{i \to i'} > 0$.
From the labeled digraph, one obtains the ODE for the species concentrations, 
\[
\dd{x}{t} = \sum_{i \to i' \in E} k_{i \to i'} \, x^{y(i)} \big(y(i')-y(i)\big) .
\]
The sum ranges over all reactions, and every summand is a product of the reaction rate and the difference of product and educt complexes.
The right-hand-side can be decomposed into stoichiometric and graphical contributions,
\[
\dd{x}{t} = Y I_E \, v_k(x) = Y A_k \, x^Y ,
\]
where $Y \in \R^{n \times V}_{\ge0}$ is the matrix of complexes,
$I_E \in \R^{V \times E}$ is the incidence matrix, 
and $A_k \in \R^{V \times V}$ is the Laplacian matrix of the digraph $G$,
labeled with the rate constants $k \in \R^E_{>0}$.
The vector of reaction rates $v_k(x) \in \R^{E}_\ge$ is defined via $(v_k(x))_{i \to i'} = k_{i \to i'} \, x^{y(i)}$,
and the vector of monomials $x^Y \in \R^V_{\ge0}$ is defined via $(x^Y)_i = x^{y(i)}$,
where $y(i)$ is the $i$-th column of $Y$.

A positive steady state $x \in \R^n_{>0}$ of the ODE
that fulfills
\[
A_k \, x^Y = 0
\]
is called a {\em complex-balanced equilibrium}. % (CBE).
Another important object is the {\em stoichiometric subspace} 
\[
S = \im (Y I_E) .
\]
Clearly, $\dd{x}{t} \in S$, and hence $x(t) \in x(0)+S$.
For $x' \in \R^n_{\ge0}$, the set $(x'+S) \cap \R^n_{\ge0}$ is called a {\em stoichiometric class}.
%
%==============================
%
The {\em deficiency} of a CRN is given by
\[
\delta 
= \dim(\ker Y \cap \im I_E) 
= m - \ell - \dim(S) ,
\]
where $m$ is the number of vertices, 
and $\ell$ is the number of connected components of the digraph.
Finally, 
a CRN is called {\em weakly reversible} if all components of the digraph are strongly connected.

Now, we can state the celebrated deficiency zero theorem for MAK, 
formulated by Horn, Jackson, and Feinberg in 1972.

\begin{thm}[$\delta=0$ theorem; cf.~\cite{HornJackson1972}, \cite{Horn1972}, and \cite{Feinberg1972}] 
For a CRN with MAK, 
there exists a unique (complex-balanced, asymptotically stable) equilibrium 
in every stoichiometric class and for all rate constants
if and only if 
$\delta=0$ and the network is weakly reversible.
\end{thm}

The $\delta=0$ theorem is a strong result. 
It characterizes CRNs with MAK that are dynamically as simple and stable as possible. 
However, MAK is an assumption that holds for elementary reactions in homogeneous and dilute solutions.
In intracellular environments, which are highly structured and crowded, 
and for reaction mechanisms,
more general kinetics are needed.
As a prominent approach,
biochemical systems theory \cite{Savageau1969,Voit2013}
proposes power laws in the species concentrations,
where the kinetic orders may differ from the stoichiometric coefficients.
In chemical reaction network theory,
power-law kinetics has been termed general(ized) mass-action kinetics (GMAK)~\cite{HornJackson1972,mueller:regensburger:2012,mueller:regensburger:2014}.
As already noted by Horn and Jackson~\cite{HornJackson1972}, every CRN with GMAK can be written as another CRN with MAK,
where the stoichiometric coefficients need not be integers.
However, the resulting network typically loses desired properties such as weak reversibility and zero deficiency.
In our more recent definition of CRNs with GMAK~\cite{mueller:regensburger:2012,mueller:regensburger:2014},
we allow for power-law kinetics, without having to rewrite the network.

%This was one motivation to study CRNs with {\em generalized} mass-action kinetics (GMAK),
%as defined in previous work \cite{mueller:regensburger:2012,mueller:regensburger:2014}.
%Another motivation was that a CRN with MAK may not have zero deficiency and may not be weakly reversible, 
%however, there may be a dynamically equivalent CRN with GMAK that has the desired properties.
In fact,
a CRN with MAK may not have zero deficiency and may not be weakly reversible,
but there may be a dynamically equivalent CRN with GMAK that has the desired properties.
In particular,
dynamical equivalence to a network having zero `effective' and `kinetic' deficiencies
allows a parametrization of all positive equilibria~\cite{JohnstonMuellerPantea2018}.
Such a parametrization can be computed by linear algebra techniques and does not require tools from algebraic
geometry such as Gr\"obner bases,
as demonstrated for the EnvZ-OmpR and shuttled WNT signaling pathways.
For algorithmic methods to identify dynamically equivalent CRNs and further applications to biochemical networks,
see~\cite{Johnston2014,Johnston2015,TonelloJohnston2018,JohnstonBurton2018}. 

Relations between biochemical systems theory and chemical reaction network theory are discussed in~\cite{ArceoJoseMarin-SanguinoMendoza2015,ArceoJoseLaoMendoza2017,TalabisArceoMendoza2018}.
Power-law systems from biochemical systems theory can be realized as CRNs with GMAK having desired properties, 
and results e.g.\ from~\cite{mueller:regensburger:2012,mueller:regensburger:2014} are applied
to models of yeast fermentation, purine metabolism~\cite{ArceoJoseMarin-SanguinoMendoza2015},
and further paradigmatic models from systems biology~\cite{ArceoJoseLaoMendoza2017}.

We continue our introduction to chemical reaction networks (with generalized mass-action kinetics).
For the reaction above,
$
1 \ce{A} + 1 \ce{B} \to \ce{C} ,
$
now under the assumption of GMAK,
the reaction rate is given by
$
v = k \, x_A^a x_B^b ,
$
where the kinetic orders $a,b \in \R$ need not coincide with the stoichiometric coefficients.
% In fact, they need not even be integers, but can be arbitrary (nonnegative) real numbers.
One writes
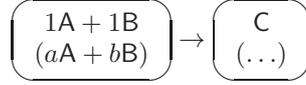
\[
\ovalbox{$\begin{array}{c} 1\ce{A}+1\ce{B} \\ (a\ce{A}+b\ce{B}) \end{array}$} 
\to 
\ovalbox{$\begin{array}{c} \ce{C} \\ (\ldots) \end{array}$} 
\]
with the kinetic-order information in brackets.
For a general reaction
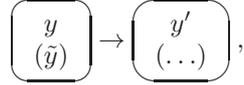
\[ 
\ovalbox{$\begin{array}{c} y \\ (\tilde y) \end{array}$}
\to 
\ovalbox{$\begin{array}{c} y' \\ (\ldots) \end{array}$} \; ,
\]
one has
\[
v = k \, x^{\tilde y} ,
\]
where $\tilde y \in \R^n$ is called a {\em kinetic(-order) complex}.

As above, a CRN is based on a digraph $G=(V,E)$, 
but now every vertex $i \in V$ is labeled with stoichiometric {\em and} kinetic-order complexes,
$y(i)$ and $\tilde y(i)$, respectively.
(And every edge is labeled with a rate constant.)
From the labeled digraph, one obtains the ODE
\[
\dd{x}{t} = \sum_{i \to i' \in E} k_{i \to i'} \, x^{\tilde y(i)} \big(y(j)-y(i)\big) .
\]

Again the right-hand-side of the ODE can be decomposed,
now into stoichiometric, graphical, and kinetic-order contributions,
\[
\dd{x}{t} = Y A_k \, x^{\tilde{Y}} ,
\]
where $\tilde Y \in \R^{n \times V}_{\ge0}$ is the matrix of kinetic-order complexes.
Accordingly, a steady state $x \in \R^n_{>0}$ that fulfills
\[
A_k \, x^{\tilde{Y}} = 0
\]
is called a complex-balanced equilibrium.
Finally, like the corresponding stoichiometric objects,
one introduces the {\em kinetic-order subspace} 
\[
\tilde S = \im (\tilde Y I_E) 
\]
and the {\em kinetic(-order) deficiency} 
\[
\tilde \delta
= \dim(\ker \tilde Y \cap \im I_E)
= m - \ell - \dim(\tilde S) .
\]

The classical $\delta=0$ theorem holds for MAK.
In previous work, we formulated a first analogue for GMAK.

\begin{thm}[$\tilde \delta=0$ theorem; cf.~\cite{mueller:regensburger:2014}] \label{thm:dt1}
For a CRN with GMAK, 
there exists a complex-balanced equilibrium 
for all rate constants
if and only if 
$\tilde \delta=0$ and the network is weakly reversible.
\end{thm}

However, this theorem does not fully correspond to the classical one
which guarantees the unique existence of a complex-balanced equilibrium in every stoichiometric class.
For GMAK, complex-balanced equilibria are determined by kinetic orders,
where\-as classes are determined by stoichiometry.
In fact, a true analogue requires extra conditions on the stoichiometric and kinetic-order subspaces,
$S$ and $\tilde S$.

For given $k \in \R^E_{>0}$,
let $Z_k$ be the set of complex-balanced equilibria,
and for given $x' \in \R^n_{>0}$,
let $(x'+S) \cap \R^n_{\ge0}$ be the corresponding stoichiometric class.
We aim to characterize existence and uniqueness of an element in the intersection
\[
Z_k \cap (x' + S)
\]
for all $x' \in \R^n_{>0}$, for all $k \in \R^E_{>0}$.
By Theorem~\ref{thm:dt1},
$Z_k \neq \emptyset$ for all $k \in \R^E_{>0}$
if and only if $\tilde \delta=0$ and the network is weakly reversible,
which we assume in the following.

By Theorem~1 in~\cite{mueller:regensburger:2014},
$x^*_k \in Z_k$ implies the exponential parametrization
\[
Z_k = x^*_k \circ \e^{\tilde S^\perp} .
\]
Moreover, for a weakly reversible CRN,
every $x^* \in \R^n_{>0}$ is a complex-balanced equilibrium for some rate constant $k \in \R^E_{>0}$,
see e.g.\ the proof of Lemma~1 in~\cite{mueller:regensburger:2014}.
Hence, we aim to characterize existence and uniqueness of an element in the intersection
\[
x^* \circ \e^{\tilde S^\perp} \cap \; (x' + S)
\]
for all $x',x^* \in \R^n_{>0}$.

%For given $k \in \R^E_{>0}$,
%let $Z_k$ be the set of complex-balanced equilibria and $x^*_k \in Z_k$.
%By Theorem~1 in~\cite{mueller:regensburger:2014},
%we have the exponential parametrization
%\[
%Z_k = x^*_k \circ \e^{\tilde S^\perp} .
%\]
%For given $x' \in \R^n_{>0}$,
%let $(x'+S) \cap \R^n_{\ge0}$ be the corresponding stoichiometric class.
%%
%We aim to characterize existence and uniqueness of an element in the intersection
%\[
%Z_k \cap (x' + S)
%\]
%for all $x' \in \R^n_{>0}$, for all $k \in \R^E_{>0}$.
%To guarantee $Z_k \neq \emptyset$, 
%we assume that $\tilde \delta=0$ and that the network is weakly reversible,
%cf.~Theorem~\ref{thm:dt1},
%and we reformulate the problem.
%
%For a weakly reversible CRN,
%every $x^* \in \R^n_{>0}$ is a complex-balanced equilibrium for some rate constant $k \in \R^E_{>0}$,
%%
%Hence, we aim to characterize existence and uniqueness of an element in the intersection
%\[
%x^* \circ \e^{\tilde S^\perp} \cap \; (x' + S)
%\]
%for all $x',x^* \in \R^n_{>0}$.

For fixed $x',x^*$, we are interested in existence and uniqueness of $u \in S$, $v \in \tilde S^\perp$ such that
\[
x^* \circ \e^v = x' + u 
\]
and introduce $W \in \R^{d \times n}, \tilde W \in \R^{\tilde d \times n}$ with full ranks $d, \tilde d \le n$ such that
\[
S = \ker W, \quad \tilde S = \ker \tilde W.
\]
We multiply with $W$, write $v = \tilde W^\trans \xi$ with $\xi \in \R^{\tilde d}$, and obtain
\[
W (x^* \circ \e^{\tilde W^\trans \xi}) = W x' .
\]
Hence, we are interested in existence and uniqueness of $\xi \in \R^{\tilde d}$ such that the last equation holds.

Finally, we note that $W x' \in C^\circ$, the interior of $C = \cone W$,
and vary over all $x' \in \R^n_{>0}$ or, equivalently, over all elements of $C^\circ$.
As a result, we aim to characterize
bijectivity of the map
\begin{align*}
F_{x^*} \colon & \R^{\tilde d} \to C^\circ \subseteq \R^d , \\
& \xi \mapsto W (x^* \circ \e^{\tilde W^\trans \xi}) = \sum_{i=1}^n x^*_i \e^{\tilde w^i \mal \xi} w^i
\end{align*}
for all $x^* \in \R^n_{>0}$,
that is, the simultaneous bijectivity of the map $F_{x^*}$ for all $x^*>0$.
%To simplify notation, one writes $c>0$ instead of $x^* \in \R^n_{>0}$ for the positive parameters of the map.
Indeed, this is the content of Theorem~\ref{bij},
and the deficiency zero theorem can be fully extended to GMAK (except for stability).

\begin{thm}[$\delta=\tilde \delta=0$ theorem] \label{thm:generalized}
For a CRN with GMAK, 
there exists a unique complex-balanced equilibrium 
in every stoichiometric class and for all rate constants
if and only if 
$\delta=\tilde\delta=0$, the network is weakly reversible,
and conditions (i), (ii), (iii) in Theorem~\ref{bij} hold.
\end{thm}

In contrast to MAK, 
where complex-balanced equilibria are asymptotically stable,
already two-species CRNs with GMAK
lead to planar systems which have a unique (complex-balanced) equilibrium,
but show rich dynamical behavior,
including super/sub-critical or degenerate Hopf bifurcations, centers, and up to three limit cycles,
see~\cite{boros:hofbauer:mueller:2017,boros:hofbauer:mueller:regensburger:2018, BorosHofbauerMuellerRegensburger2018b,BorosHofbauer2018}.

By Theorem~\ref{pert_tildeS} (and the problem derivation given above),
Theorem~\ref{thm:generalized} is robust
with respect to small perturbations of the kinetic orders
if and only if the closure condition $\sign(S) \subseteq \overline{\sign(\tilde S)}$ holds.

\begin{thm}[robust $\delta=\tilde \delta=0$ theorem] \label{thm:generalizedrobust}
For a CRN with GMAK, 
there exists a unique complex-balanced equilibrium 
in every stoichiometric class, for all rate constants,
and for all small perturbations of the kinetic orders
if and only if 
$\delta=\tilde\delta=0$, the network is weakly reversible,
and $\sign(S) \subseteq \overline{\sign(\tilde S)}$.
\end{thm}

For a CRN with MAK, the stoichiometric and kinetic-order subspaces agree, that is,
$S = \tilde S$, and obviously $\sign(S) \subseteq \overline{\sign(\tilde S)}$.
Hence, the classical deficiency zero theorem for MAK is robust
with respect to small perturbations of the kinetic orders (from the stoichiometric coefficients).

\begin{cor}[robust $\delta=0$ theorem] \label{cor:robust}
For a CRN with MAK, 
there exists a unique (complex-balanced, asymptotically stable) equilibrium 
in every stoichiometric class,
for all rate constants,
and for all small perturbations of the kinetic orders (from the stoichiometric coefficients)
if and only if 
$\delta=0$ and the network is weakly reversible.
\end{cor}

\subsection*{Acknowledgments}

The closure condition (\ref{cc}) was suggested by Gheorghe Craciun
as a criterion for the bijectivity of the family of exponential maps.
We thank Gheorghe Craciun, Casian Pantea, and Polly Yu for fruitful discussions 
(at workshops at 
the University of Wisconsin-Madison in 2015,
the American Institute of Mathematics, San Jose, in 2016, 
and the Banff International Research Station
and the Mathematisches Forschungsinstitut Oberwolfach in 2017).
We also thank three anonymous referees for their careful reading and numerous helpful comments.

SM was supported by the Austrian Science Fund (FWF), project P28406.
GR was supported by the FWF, project P27229.

%%%%%%%%% %%%%%%%%% %%%%%%%%% %%%%%%%%% %%%%%%%%%
%%%%%%%%% %%%%%%%%% %%%%%%%%% %%%%%%%%% %%%%%%%%%
%%%%%%%%% %%%%%%%%% %%%%%%%%% %%%%%%%%% %%%%%%%%%

\appendix
% \addcontentsline{toc}{section}{Appendices}
\section*{Appendices}

\section{Sign vectors and face lattices} \label{sv}

In the context of (realizable) oriented matroids,
we discuss the relation between sign vectors of linear subspaces
and face lattices of polyhedral cones. % and polytopes.
For further details, we refer to~\cite[Chapter~7]{BachemKern1992}, \cite[Chapters~2 and 6]{Ziegler1995}, 
\cite{Richter-GebertZiegler1997}, 
and the encyclopedic study~\cite{BjornerLasSturmfelsWhiteZiegler1999}.

Let $W = (w^1, \ldots , w^n) \in \R^{d\times n}$ with $d \le n$ have full rank.
Then $W$ is called a {\em vector configuration} (of $n$ vectors in $\R^d$),
and $\im W^\trans \subseteq \R^n$ is a corresponding linear subspace.
Now let $v = W^\trans x \in \im W^\trans$ with $x \in \R^d$.
Then $v_i = w^i \mal x$,
and the sign vector $\tau = \sign(v) \in \sign(\im W^\trans) \subseteq \{-,0,+\}^n$ describes the positions of the vectors $w^1, \ldots, w^n$
relative to the hyperplane with normal vector $x$.

Elements of $\sign(\im W^\trans)$ are called {\em co\-vectors},
and elements of $\sign(\im W^\trans)$ with minimal support are called {\em co\-circuits}.
Analogously, elements of $\sign(\ker W)$ are called {\em vectors},
and elements of $\sign(\ker W)$ with minimal support are called {\em circuits}.

The {\em chirotope} of the vector configuration $W$ is the map
\begin{align*}
\chi \colon & \{1,\ldots,n\}^d \to \{-,0,+\} \, , \\
& (i_1, \ldots, i_d) \mapsto \sign(\det(w^{i_1}, \ldots , w^{i_d}))
\end{align*}
which records for each $d$-tuple of vectors $w^i$
if it forms a positively (or negatively) oriented basis of $\R^d$ or it is not a basis.
%It can, for example, be used to test algorithmically if the sign vectors of two subspaces are equal,
%that is, to decide if $\sign(\im(V)) = \sign(\im(\tilde{V}))$
%for two matrices $V, \tilde{V} \in \R^{n \times d}$.

The {\em oriented matroid} of $W$ is a combinatorial structure that can be given by any of the above data (co/vectors, co/circuits, or chirotopes) 
and defined/characterized in terms of any of the corresponding axiom systems.
As an example, we state the chirotope/cocircuit translation,
see Theorems 6.2.3 in \cite{Richter-GebertZiegler1997}
or 8.1.6 in \cite{DeLoeraRambauSantos2010}.

\begin{thm} \label{translation}
Let $W \in \R^{d \times n}$ be a vector configuration with chirotope $\chi$.
Then the set of cocircuits is given by
\[
\mathcal{C}^*(\chi) =
\Big\{ \big( \chi(I,1), \chi(I,2), \ldots, \chi(I,n) \big) \mid I \in \{1,\ldots,n\}^{d-1} \Big\} .
\]
Conversely, let $W \in \R^{d \times n}$ be a vector configuration with cocircuits $\mathcal{C}^*$.
Then there exists a unique pair of chirotopes $(\chi,-\chi)$
such that $\mathcal{C}^*(\chi) = \mathcal{C}^*(-\chi) = \mathcal{C}^*$.
\end{thm}

The {\em face lattice} of $C = \cone W \subseteq \R^d$, the polyhedral cone generated by the vectors $w^1, \ldots, w^n$,
can be obtained from the sign vectors of the linear subspace $\im W^\trans$.
In fact, it is the set $\sign(\im W^\trans)_\zp = \sign(\im W^\trans) \cap \{0,+\}^n$ with the partial order induced by the relation $+>0$.
A face $f$ of $C$ corresponds to a supporting hyperplane with normal vector $x$
such that $w^i \mal x = 0$ for $w^i \in f$
and $w^i \mal x > 0$ for $w^i \not\in f$, lying on the positive side of the hyperplane.
(The vector~$x$ lies on the corresponding face of the dual cone $C^*$.)
Hence the face $f$ with $I=\{ i \mid w^i \in f \}$ is characterized by the sign vector $\tau = \sign(W^\trans x) \in \sign(\im W^\trans)_\zp$ with $I=\tau^0$. Moreover, for two faces $f$ and $f'$ of $C$ with corresponding nonnegative sign vectors $\tau$ and $\tau'$, the order is reversed: $f \subseteq f'$ if and only if $\tau' \leq \tau$.

The {\em lineality space} of a cone $C$ is given by the set $C \cap (-C)$.
It is the minimal face of $C$, in the sense that it is contained in all faces.
The lineality space of $C = \cone W$ 
is characterized by the maximal element of $\sign(\im W^\trans)_\zp$
or, equivalently, by the maximal element of $\sign(\ker W)_\zp$.
Thereby, nonzero elements of $\sign(\ker W)_\zp$ correspond to positive dependencies of vectors $w^i$ (in the lineality space).

A cone $C$ is called {\em pointed} if its lineality space is $\{0\}$, % if $C \cap (-C) = \{ 0 \}$
that is, if it has vertex~0. 
%A cone is pointed if and only if it has an extreme ray, 
%and every pointed polyhedral cone is the conical hull of its finitely many extreme rays.
Note that, if $(+,\ldots,+)^\trans \in \sign(\im W^\trans)_\zp$ (that is, $\sign(\ker W)_\zp = \{0\}$),
then $C = \cone W$ is pointed.

Finally, we note that sign vectors of a linear subspace are closed under composition:
Let $S$ be a subspace of $\R^n$ and $\tau,\rho \in \sign(S)$.
Then, also $\tau \circ \rho \in \sign(S)$.
To see this, let $u,v \in \R^n$ with $\tau = \sign(u)$, $\rho = \sign(v)$.
Then, $\tau \circ \rho = \sign(u+\eps v) \in \sign(S)$ for small $\eps>0$.
Moreover, every nonzero sign vector of a linear subspace can be written as a conformal composition
of sign vectors with minimal support,
see Theorem~1 in \cite{Rockafellar1969}, Proposition~5.35 in \cite{BachemKern1992}, or Theorem~3 in \cite{MuellerRegensburger2016}.

\begin{thm} \label{conformal}
Let $S$ be a subspace of $\R^n$ and $\tau \in \sign(S)$ be nonzero.
Then there are $\rho_i \in \sign(S)$ with minimal support and $\rho_i \le \tau$ such that
\[
\tau = \rho_1 \circ \cdots \circ \rho_N .
\]
The $\rho_i$ can be chosen such that $N \le \min (\dim(S), \abs{\supp(\tau)})$.
\end{thm}

%%%%%%%%% %%%%%%%%% %%%%%%%%% %%%%%%%%% %%%%%%%%%

\section{A general theorem of the alternative} \label{ta}

We recall a general theorem of the alternative for subspaces of $\R^n$ 
that allows to easily derive theorems of the alternative for sign vectors of a linear subspace and its orthogonal complement. 
For the relation to standard theorems of the alternative,
see \cite{Minty1974};
for the corresponding statements for arbitrary oriented matroids,
see \cite[Section 3.4]{BjornerLasSturmfelsWhiteZiegler1999} or~\cite[Chapter~5]{BachemKern1992}.

\begin{dfn}
Let $x \in \R^n$, and let $I_1,\ldots,I_n$ be intervals of $\R$.
We define the interval
\begin{align*}
I(x) &\equiv x_1 I_1  + \ldots + x_n I_n \\ &= \{ x_1 y_1 + \ldots + x_n y_n \in \R \mid y_1 \in I_1, \ldots, y_n \in I_n \}
\end{align*}
and write $I(x) > 0$ if $y>0$ for all $y \in I(x)$.
\end{dfn}

\begin{thm}[Theorem 22.6 in \cite{Rockafellar1970}] \label{minty}
Let $S$ be a subspace of $\R^n$, and let $I_1,\ldots,I_n$ be intervals of $\R$.
Then one and only one of the following alternatives holds:
\begin{itemize}
\item[(a)]
There exists a vector $x = (x_1,\ldots,x_n)^\trans \in S$ such that
\[
x_1 \in I_1, \, \ldots,\, x_n \in I_n .
\]
\item[(b)]
There exists a vector $\xa = (\xa_1,\ldots, \xa_n)^\trans \in S^\perp$ such that
\[
\xa_1 I_1 + \ldots + \xa_n I_n > 0 .
\]
\end{itemize}
\end{thm}

\begin{cor} \label{minty1}
Let $S$ be a subspace of $\R^n$ and $\sigma \in \{ -,0,+ \}^n$ be a nonzero sign vector.
Then either (a) there exists a vector $x \in S$ with $x_i > 0$ for $i \in \sigma^+$ and $x_i < 0$ for $i \in \sigma^-$
or (b) there exists a nonzero vector $\xa \in S^\perp$ with $\xa_i \ge 0$ for $i \in \sigma^+$, $\xa_i \le 0$ for $i \in \sigma^-$, and $\xa_i=0$ otherwise.
In terms of sign vectors, either there exists $\xi \in \sign(S)$ with $\xi \ge \sigma$
or there exists a nonzero $\xi^* \in \sign(S^\perp)$ with $\xi^* \le \sigma$.
\end{cor}
\begin{proof}
By Theorem~\ref{minty} with $I_i = (0,+\infty)$ for $i \in \sigma^+$, $I_i = (-\infty,0)$ for $i \in \sigma^-$, and $I_i = (-\infty,+\infty)$ otherwise.
\end{proof}

\begin{cor} \label{orthogonal}
Let $S$ be a subspace of $\R^n$. Then,
\[
\sign(S^\perp)=\sign(S)^\perp. 
\]
\end{cor}
\begin{proof}
($\subseteq$): Let $\tau \in \sign(S^\perp)$ and $\rho \in \sign(S)$.
Now, let $u \in S^\perp$ and $v \in S$ such that $\tau = \sign(u)$ and $\rho = \sign(v)$.
Then, $u \cdot v = 0$ implies $\tau \cdot \rho = 0$,
and hence $\tau \in \sign(S)^\perp$.

($\supseteq$): Let $\tau \notin \sign(S^\perp)$,
that is, there exists no $x \in S^\perp$ such that $\sign(x) = \tau$.
By Theorem~\ref{minty} with $I_i = (0,+\infty)$ for $i \in \tau^+$, $I_i = (-\infty,0)$ for $i \in \tau^-$, and $I_i = \{0\}$ otherwise, there exists a nonzero $\xa \in S$ such that 
$\xa_i \ge 0$ for $i \in \tau^+$ and $\xa_i \le 0$ for $i \in \tau^-$.
Let $\rho = \sign(\xa) \in \sign(S)$. 
Then, $\tau \cdot \rho \neq 0$,
and hence $\tau \notin \sign(S)^\perp$.
\end{proof}
For an alternative proof, using Farkas Lemma, see Proposition~6.8 in \cite{Ziegler1995}.

%%%%%%%%% %%%%%%%%% %%%%%%%%% %%%%%%%%% %%%%%%%%%

\bibliographystyle{abbrv} %siamplain
\bibliography{fractional,birch}

\begin{thebibliography}{10}

\bibitem{AlbouyFu2014}
A.~Albouy and Y.~Fu.
\newblock Some remarks about {D}escartes' rule of signs.
\newblock {\em Elem. Math.}, 69, 2014.

\bibitem{ArceoJoseLaoMendoza2017}
C.~P.~P. Arceo, E.~C. Jose, A.~R. Lao, and E.~R. Mendoza.
\newblock Reaction networks and kinetics of biochemical systems.
\newblock {\em Math. Biosci.}, 283:13--29, 2017.

\bibitem{ArceoJoseMarin-SanguinoMendoza2015}
C.~P.~P. Arceo, E.~C. Jose, A.~Marin-Sanguino, and E.~R. Mendoza.
\newblock Chemical reaction network approaches to biochemical systems theory.
\newblock {\em Math. Biosci.}, 269:135--152, 2015.

\bibitem{BachemKern1992}
A.~Bachem and W.~Kern.
\newblock {\em Linear programming duality}.
\newblock Springer-Verlag, Berlin, 1992.

\bibitem{BanachMazur1934}
S.~{Banach} and S.~{Mazur}.
\newblock {\"Uber mehrdeutige stetige Abbildungen.}
\newblock {\em {Stud. Math.}}, 5:174--178, 1934.

\bibitem{BP2016}
M.~Banaji and C.~Pantea.
\newblock Some results on injectivity and multistationarity in chemical
  reaction networks.
\newblock {\em SIAM J. Appl. Dyn. Syst.}, 15:807--869, 2016.

\bibitem{BjornerLasSturmfelsWhiteZiegler1999}
A.~Bj{\"o}rner, M.~Las~Vergnas, B.~Sturmfels, N.~White, and G.~M. Ziegler.
\newblock {\em Oriented matroids}, volume~46 of {\em Encyclopedia Math. Appl.}
\newblock Cambridge University Press, Cambridge, second edition, 1999.

\bibitem{BorosHofbauer2018}
B.~Boros and J.~Hofbauer.
\newblock Planar {S}-systems: Permanence.
\newblock {\em J. Differential Equations}, 266:3787--3817, 2019.

\bibitem{boros:hofbauer:mueller:2017}
B.~Boros, J.~Hofbauer, and S.~M{\"u}ller.
\newblock On global stability of the {L}otka reactions with generalized
  mass-action kinetics.
\newblock {\em Acta Appl. Math.}, 151:53--80, 2017.

\bibitem{boros:hofbauer:mueller:regensburger:2018}
B.~Boros, J.~Hofbauer, S.~M{\"u}ller, and G.~Regensburger.
\newblock The center problem for the {L}otka reactions with generalized
  mass-action kinetics.
\newblock {\em Qual. Theory Dyn. Syst.}, 17:403--410, 2018.

\bibitem{BorosHofbauerMuellerRegensburger2018b}
B.~Boros, J.~Hofbauer, S.~M{\"u}ller, and G.~Regensburger.
\newblock Planar {S}-systems: Global stability and the center problem.
\newblock {\em Discrete Contin. Dyn. Syst. Ser. A}, 2(29):707--727, 2019.

\bibitem{Chaiken1996}
S.~Chaiken.
\newblock Oriented matroid pairs, theory and an electric application.
\newblock In {\em Matroid theory ({S}eattle, {WA}, 1995)}, volume 197 of {\em
  Contemp. Math.}, pages 313--331. Amer. Math. Soc., Providence, RI, 1996.

\bibitem{CraciunDickensteinShiuSturmfels2009}
G.~Craciun, A.~Dickenstein, A.~Shiu, and B.~Sturmfels.
\newblock Toric dynamical systems.
\newblock {\em J. Symbolic Comput.}, 44:1551--1565, 2009.

\bibitem{CraciunFeinberg2005}
G.~Craciun and M.~Feinberg.
\newblock Multiple equilibria in complex chemical reaction networks. {I}. {T}he
  injectivity property.
\newblock {\em SIAM J. Appl. Math.}, 65:1526--1546, 2005.

\bibitem{CraciunGarcia-PuenteSottile2010}
G.~Craciun, L.~Garcia-Puente, and F.~Sottile.
\newblock Some geometrical aspects of control points for toric patches.
\newblock In M.~D\ae{}hlen, M.~S. Floater, T.~Lyche, J.-L. Merrien, K.~Morken,
  and L.~L. Schumaker, editors, {\em Mathematical Methods for Curves and
  Surfaces}, volume 5862 of {\em Lecture Notes in Comput. Sci.}, pages
  111--135, Heidelberg, 2010. Springer.

\bibitem{CMPY2018}
G.~{Craciun}, S.~{M{\"u}ller}, C.~{Pantea}, and P.~Y. {Yu}.
\newblock {A generalization of Birch's theorem and vertex-balanced steady
  states for generalized mass-action systems}.
\newblock 2018.
\newblock \href{https://arxiv.org/abs/1802.06919}{arXiv:1802.06919} [math.DS].

\bibitem{DeLoeraRambauSantos2010}
J.~A. {De Loera}, J.~{Rambau}, and F.~{Santos}.
\newblock {\em {Triangulations. Structures for algorithms and applications.}}
\newblock Berlin: Springer, 2010.

\bibitem{Feinberg1972}
M.~Feinberg.
\newblock Complex balancing in general kinetic systems.
\newblock {\em Arch. Rational Mech. Anal.}, 49:187--194, 1972/73.

\bibitem{FW2012}
E.~Feliu and C.~Wiuf.
\newblock Preclusion of switch behavior in networks with mass-action kinetics.
\newblock {\em Appl. Math. Comput.}, 219:1449--1467, 2012.

\bibitem{Fulton1993}
W.~Fulton.
\newblock {\em Introduction to toric varieties}, volume 131 of {\em Ann. of
  Math. Stud.}
\newblock Princeton University Press, Princeton, NJ, 1993.

\bibitem{GlazmanLjubic1974}
I.~M. Glazman and J.~I. Ljubi\v{c}.
\newblock {\em Finite-dimensional linear analysis: a systematic presentation in
  problem form}.
\newblock The M.I.T. Press, Cambridge, Mass.-London, 1974.

\bibitem{G2009}
G.~Gnacadja.
\newblock Univalent positive polynomial maps and the equilibrium state of
  chemical networks of reversible binding reactions.
\newblock {\em Adv. in Appl. Math.}, 43:394--414, 2009.

\bibitem{G2012}
G.~Gnacadja.
\newblock A {J}acobian criterion for the simultaneous injectivity on positive
  variables of linearly parameterized polynomial maps.
\newblock {\em Linear Algebra Appl.}, 437:612--622, 2012.

\bibitem{GopalkrishnanMillerShiu2014}
M.~Gopalkrishnan, E.~Miller, and A.~Shiu.
\newblock A geometric approach to the global attractor conjecture.
\newblock {\em SIAM J. Appl. Dyn. Syst.}, 13:758--797, 2014.

\bibitem{Gordon1972}
W.~B. {Gordon}.
\newblock {On the diffeomorphisms of Euclidean space.}
\newblock {\em {Am. Math. Mon.}}, 79:755--759, 1972.

\bibitem{Hadamard1906}
J.~{Hadamard}.
\newblock {Sur les transformations ponctuelles.}
\newblock {\em {Bull. Soc. Math. Fr.}}, 34:71--84, 1906.

\bibitem{Horn1972}
F.~Horn.
\newblock Necessary and sufficient conditions for complex balancing in chemical
  kinetics.
\newblock {\em Arch. Rational Mech. Anal.}, 49:172--186, 1972/73.

\bibitem{HornJackson1972}
F.~Horn and R.~Jackson.
\newblock General mass action kinetics.
\newblock {\em Arch. Rational Mech. Anal.}, 47:81--116, 1972.

\bibitem{Jameson2006}
G.~J.~O. Jameson.
\newblock Counting zeros of generalised polynomials: Descartes' rule of signs
  and {L}aguerre's extensions.
\newblock {\em Math. Gaz.}, 90:223--234, 2006.

\bibitem{Johnston2014}
M.~D. Johnston.
\newblock Translated chemical reaction networks.
\newblock {\em Bull. Math. Biol.}, 76:1081--1116, 2014.

\bibitem{Johnston2015}
M.~D. Johnston.
\newblock A computational approach to steady state correspondence of regular
  and generalized mass action systems.
\newblock {\em Bull. Math. Biol.}, 77:1065--1100, 2015.

\bibitem{JohnstonBurton2018}
M.~D. Johnston and E.~Burton.
\newblock Computing weakly reversible deficiency zero network translations
  using elementary flux modes.
\newblock 2018.
\newblock Submitted, \href{https://arxiv.org/abs/1808.09059}{arXiv:1808.09059}
  [math.OC].

\bibitem{JohnstonMuellerPantea2018}
M.~D. Johnston, S.~M{\"u}ller, and C.~Pantea.
\newblock A deficiency-based approach to parametrizing positive equilibria of
  biochemical reaction systems.
\newblock {\em Bull. Math. Biol.}, 81:1143--1172, 2019.

\bibitem{Khovanskii1991}
A.~G. Khovanski{\u{i}}.
\newblock {\em Fewnomials}, volume~88 of {\em Translations of Mathematical
  Monographs}.
\newblock American Mathematical Society, Providence, RI, 1991.

\bibitem{Laguerre1883}
E.~N. Laguerre.
\newblock M{\'e}moire sur la th{\'e}orie des {\'e}quations num{\'e}riques.
\newblock {\em J. Math. Pures et Appl. (3)}, 9:99--146, 1883.

\bibitem{Minty1974}
G.~J. {Minty}.
\newblock {A ``from scratch'' proof of a theorem of Rockafellar and Fulkerson.}
\newblock {\em {Math. Program.}}, 7:368--375, 1974.

\bibitem{MFR2016}
S.~M{\"u}ller, E.~Feliu, G.~Regensburger, C.~Conradi, A.~Shiu, and
  A.~Dickenstein.
\newblock Sign conditions for injectivity of generalized polynomial maps with
  applications to chemical reaction networks and real algebraic geometry.
\newblock {\em Found. Comput. Math.}, 16:69--97, 2016.

\bibitem{mueller:regensburger:2012}
S.~M\"uller and G.~Regensburger.
\newblock Generalized mass action systems: {C}omplex balancing equilibria and
  sign vectors of the stoichiometric and kinetic-order subspaces.
\newblock {\em SIAM J. Appl. Math.}, 72:1926--1947, 2012.

\bibitem{mueller:regensburger:2014}
S.~M\"uller and G.~Regensburger.
\newblock Generalized mass-action systems and positive solutions of polynomial
  equations with real and symbolic exponents.
\newblock In V.~P. Gerdt, W.~Koepf, E.~W. Mayr, and E.~H. Vorozhtsov, editors,
  {\em Computer Algebra in Scientific Computing. Proceedings of the 16th
  International Workshop (CASC 2014)}, volume 8660 of {\em Lecture Notes in
  Comput. Sci.}, pages 302--323, Cham, 2014. Springer.

\bibitem{MuellerRegensburger2016}
S.~M\"uller and G.~Regensburger.
\newblock Elementary vectors and conformal sums in polyhedral geometry and
  their relevance for metabolic pathway analysis.
\newblock {\em Front. Genet.}, 7(90):11 pages, 2016.

\bibitem{PachterSturmfels2005}
L.~Pachter and B.~Sturmfels.
\newblock Statistics.
\newblock In {\em Algebraic statistics for computational biology}, pages 3--42.
  Cambridge Univ. Press, New York, 2005.

\bibitem{Richter-GebertZiegler1997}
J.~Richter-Gebert and G.~M. Ziegler.
\newblock Oriented matroids.
\newblock In {\em Handbook of discrete and computational geometry}, pages
  111--132. CRC, Boca Raton, FL, 1997.

\bibitem{Rockafellar1969}
R.~T. Rockafellar.
\newblock The elementary vectors of a subspace of {$R^{N}$}.
\newblock In {\em Combinatorial {M}athematics and its {A}pplications ({P}roc.
  {C}onf., {U}niv. {N}orth {C}arolina, {C}hapel {H}ill, {N}.{C}., 1967)}, pages
  104--127. Univ. North Carolina Press, Chapel Hill, N.C., 1969.

\bibitem{Rockafellar1970}
R.~T. Rockafellar.
\newblock {\em Convex analysis}.
\newblock Princeton University Press, Princeton, N.J., 1970.

\bibitem{Savageau1969}
M.~A. Savageau.
\newblock {B}iochemical systems analysis. {I}. {S}ome mathematical properties
  of the rate law for the component enzymatic reactions.
\newblock {\em J. Theor. Biol.}, 25:365--369, 1969.

\bibitem{Sotille2011}
F.~Sottile.
\newblock {\em Real Solutions to Equations from Geometry}.
\newblock American Mathematical Society, Providence, RI, 2011.

\bibitem{Struik1969}
D.~J. Struik, editor.
\newblock {\em {A source book in mathematics, 1200-1800}}.
\newblock {Source Books in the History of the Sciences. Cambridge, Mass.:
  Harvard University Press, XIV}, 1969.

\bibitem{Sturmfels2002}
B.~Sturmfels.
\newblock {\em Solving systems of polynomial equations}.
\newblock CBMS Regional Conf. Ser. in Math. Conference Board of the
  Mathematical Sciences, Washington, DC, 2002.

\bibitem{TalabisArceoMendoza2018}
D.~A. S.~J. Talabis, C.~P.~P. Arceo, and E.~R. Mendoza.
\newblock Positive equilibria of a class of power-law kinetics.
\newblock {\em J. Math. Chem.}, 56:358--394, 2018.

\bibitem{TonelloJohnston2018}
E.~Tonello and M.~D. Johnston.
\newblock Network {T}ranslation and {S}teady-{S}tate {P}roperties of {C}hemical
  {R}eaction {S}ystems.
\newblock {\em Bull. Math. Biol.}, 80:2306--2337, 2018.

\bibitem{Voit2013}
E.~O. Voit.
\newblock {Biochemical systems theory: a review.}
\newblock {\em {ISRN Biomath.}}, 2013.
\newblock Article ID 897658.

\bibitem{Ziegler1995}
G.~M. Ziegler.
\newblock {\em Lectures on polytopes}.
\newblock Springer-Verlag, New York, 1995.

\end{thebibliography}

\end{document}